\documentclass{amsart}

\usepackage{url, xcolor, tikz-cd, enumerate}
\usepackage{comment}
\usepackage{amssymb}

\newtheorem{theorem}{Theorem}[section]
\newtheorem{proposition}[theorem]{Proposition}
\newtheorem{lemma}[theorem]{Lemma}

\newtheorem{fact}[theorem]{Fact}

\theoremstyle{definition}
\newtheorem{definition}[theorem]{Definition}

\theoremstyle{remark}
\newtheorem{remark}[theorem]{Remark}
\newtheorem{example}[theorem]{Example}

\newtheorem{notation}[theorem]{Notation}

\numberwithin{equation}{section}

\def\Ind{\setbox0=\hbox{$x$}\kern\wd0\hbox to 0pt{\hss$\mid$\hss} \lower.9\ht0\hbox to 0pt{\hss$\smile$\hss}\kern\wd0} 

\def\Notind{\setbox0=\hbox{$x$}\kern\wd0\hbox to 0pt{\mathchardef \nn=12854\hss$\nn$\kern1.4\wd0\hss}\hbox to 0pt{\hss$\mid$\hss}\lower.9\ht0 \hbox to 0pt{\hss$\smile$\hss}\kern\wd0}

\def \d {\delta}
\def \DD {\mathcal D}
\def \D {\Delta}

\def \U {\mathcal U}
\def \dcl {dcl}
\def \acl {acl}

\def \tp {tp}
\def \defin {\operatorname{def}}
\def \HH {\operatorname{H}}
\def \ZZ {\operatorname{Z}}
\def \diff {\operatorname{diff}}
\def \alg {\operatorname{alg}}

\def \l {\langle}
\def \r {\rangle}

\def \s {\sigma}

\def \Aut {\text{Aut}}
\def \DCF {\operatorname{DCF}}
\def \ACF {\operatorname{ACF}}

\title{Parameterized D-torsors in differential Galois theory}

\author{Omar Le\'on S\'anchez}
\address{Omar Le\'on S\'anchez\\ University of Manchester\\ School of Mathematics\\ Oxford Road \\ Manchester, M13 9PL.} 
\email{omar.sanchez@manchester.ac.uk}
\thanks{Omar Le\'on S\'anchez was partially supported by EPSRC grant EP/V03619X/1}

\author{David Meretzky}
\address{David Meretzky\\ Universidad de Los Andes\\ Departamento de Matematicas\\ Carrera 1 \#18A - 12\\
Edificio H \\ Bogotá - Colombia 111711} 
\email{d.meretzky@uniandes.edu.co}
\date{\today}

\begin{document}

\begin{abstract}
In the context of differential fields of characteristic zero with several commuting derivations, we discuss the notion of $\#$-differential equations on parameterized D-torsors and their associated Galois extensions. Using model-theoretic methods, we observe that any generalized strongly normal extension (in the sense of Pillay \cite{Pillay1} and, more generally, Le\'on S\'anchez \cite{LS1}) is the Galois extension of a parameterized D-torsor. 

%\textcolor{red}{Does this come before now?} 
Furthermore, we prove a parameterized version of a theorem of Kolchin on differential cohomology, itself of independent interest, and use it to provide a necessary and sufficient cohomological condition for when a generalized strongly normal extension is the Galois extension for a log-differential equation on its Galois group (as a parameterized D-group). We also present general model-theoretic versions of some of the main results. %\textcolor{red}{Should we move section 2 some model theoretic generalities to the end of the paper? Do we use the results of section 2 later?}

%{\color{blue} OLS: I moved the parameterized version of Kolchin's cohomology theorem to section 4 and then used this to write the main statement in section 5 in a cleaner format, with just one map in cohomology (I explained the two possible maps in section 4). We do use the model theoretic generalities in the other sections, so the model-theory should stay in section 2.}
\end{abstract}

\maketitle

\tableofcontents

\section{Introduction}

In classical Galois theory, a (finite) Galois extension of fields $L/K$ is commonly defined as a finite separable extension which is normal. It is then proven, see \cite[\S4.3]{Hungerford} for instance, that this is equivalent to asking that $L$ be the splitting field of a separable polynomial over $K$. This equivalence, together with the Galois correspondence, are fundamental results that play a key role in Galois' theorem on radical solvability of polynomial equations.

In the context of differential fields and differential equations (in characteristic zero and commuting derivations), several ``Galois theories" have been developed; some suitable for studying linear PDEs \cite{CaSi,GiGoOv,PuSi} usually referred as the (parameterized) Picard-Vessiot theory; and some others being far-reaching generalizations of the PV theory to include extensions that enjoy solely 
%\textcolor{red}{Sounds better to me to remove ``of" in ``of a key property" and replace the first dashes -- with commas and the second dash with starting a new sentence} 
a key property: some form of (strong) normality, see for example \cite{Kolchin1,Landesman,LSPillay,Pillay1}. The generalized strongly normal theory, as it appears in \cite[\S3.5]{LSPillay}, has the most general form (to the authors' knowledge); in the sense that the differential extensions with full Galois correspondences considered in all other ``differential Galois theories" are examples of generalized strongly normal extensions. 

Briefly, working inside an ambient differentially closed field $(\U,\Pi)\models \DCF_{0,m}$, a differential field extension $L/K$ is a generalized strongly normal extension if $L$ is finitely (differentially) generated over $K$ and there is a $K$-definable set $X$ in the structure $(\U,\Pi)$ such that (i) for all $\sigma\in\Aut_\Pi(\U/K)$ we have $\sigma(L)\subset L\langle X\rangle$, this is called \emph{strong normality} of $L$; and (ii) $K\langle X\rangle\cap L=K$, one could call this the \emph{no new constants} condition (aka \emph{weak orthogonality}). Here $L\langle X\rangle$ denotes the differential field generated by all the entries of $X=X(\U)$ over $L$, similarly for $K\langle X\rangle$.

The purpose of this paper is to establish a suitable analogue of the fundamental result in (algebraic) Galois theory mentioned above to the context of generalized strongly normal extensions. In rough terms, the main result of the paper in Section~\ref{sec:mainresult} shows that any such extension is the ``splitting differential field" of 
%\textcolor{red}{``the sharp equation" instead of ``a sharp equation"?} 
the $\#$-differential equation on a parameterized D-torsor. We will define these notions rigorously in Section~\ref{sec:defDvar}.  

Instances of our results were shown by Kolchin in the Picard-Vessiot and strongly normal setting in \cite{Kolchin1} and proved in the finite-dimensional (aka $\Pi$-type zero) generalized strongly normal setting by Pillay under additional assumptions in \cite{Pillay1,Pillay2}. In the infinite-dimensional case, versions of our results appear (under similar additional assumptions) in work of the first author \cite{LS1}. The arguments in these papers use Galois-cohomological machinery to prove what Kolchin and Pillay refer to as ``G-primitive" style results. In Section \ref{sec:modPrelims}, we collect from the literature the model-theoretic generalizations of the cohomological results necessary for both our main theorem and a model-theoretic generalization.

%{\bf I need to rewrite the rest of the introduction. Perhaps as follows: First introduce D-varieties, then mention that this category admits products and so makes sense to talk about D-groups and D-torsors (i.e., finite rank/dimensional objects in DCF$_{0,m}$). Give some details and then talk about the log derivative on a D-torsor $(T,t)$ and D-points $(T,t)^\#$. Then define log equations and their Galois extensions, and correspondence. Then mention an infinite dimensional example and introduce the parametric version of things, and then state the main result as Theorem A. Then comment that certain special cases do appear in the literature. At the end describe what we do in each section.}

%\
We now describe in more detail our results and their place within the scope of the literature mentioned above. 
%We fix a ground differential field $(K,\D)$ which is small in our ambient $(\U,\D)$. 
Recall that an algebraic D-group over $K$ is an algebraic group $G$ equipped with a regular section $t$ of the prolongation (or twisted tangent bundle) $\tau_\Pi G$, both over $K$, such that $t:G\to\tau_\Pi G$ is a group homomorphism. This induces a differential $K$-algebra structure on the structure sheaf $\mathcal O_G$ and we further require that these derivations commute. The $\#$-differential equation on $(G,t)$ is defined as
$$\nabla_\Pi(x)=t(x)$$
where $\nabla_\Pi(x)=(x,D_1 x,\dots,D_m x)$ with $\Pi=\{D_1,\dots,D_m\}$, and the set of $\#$-points of $(G,t)$, denoted $(G,t)^\#$, is the set of solutions in $G$ of the $\#$-differential equation. Note that $(G,t)^\#$ is a definable group in $(\U,\Pi)$ or equivalently a $\Pi$-algebraic group; moreover, it is finite-dimensional in the sense that its $\Pi$-type is zero (in other words, the differential field generated by a generic point has finite transcendence degree over $K$). Let $e$ be the identity of $G$ and $\tau_\Pi G_e$ the fibre of $\tau_\Pi G$ over $e$ (this can be identified with the $m$-fold product of the Lie algebra of $G$). For $\alpha$ a $K$-point in $\tau_\Pi G_e$,  the $\alpha$-log differential equation on $(G,t)$ is defined as
\begin{equation}\label{logeq}
\nabla_\Pi(x)=\alpha\cdot t(x)
\end{equation}
where the product $\alpha$ with $t(x)$ is computed in the algebraic group $\tau_\Pi G$.

%The logarithmic derivative $\ell_s:G\to \tau_\Pi G_e$ of $(G,s)$ is then defined as
%$$\ell_s(g):=\nabla_\Pi(g)\cdot (s(g))^{-1}$$
 %and the product and inverse are computed in $\tau_\D G$. The kernel of $\ell_s$ is denoted by $G^\#$; note that this is a finite-dimensional differential algebraic group definable in the structure $(\U,\D)$. A logarithmic differential equation on $(G,s)$ is an equation of the form
%\begin{equation}\label{logeq}
%\ell_s(y)=a, \quad \text{ where $y$ ranges in $G$}
%\end{equation}
%and $a$ is a $K$-point of $\tau_\D G_e$. 
By a (differential) Galois extension over $K$ for the $\alpha$-log differential equation~\eqref{logeq} we mean a differential field $L$ of the form $K\langle a\rangle$ for some solution $a$ of the equation such that $K\langle (G,t)^\#\rangle\cap L=K$, where recall that $K\langle (G,t)^\# \rangle$ denotes the differential field generated by all the entries of $(G,t)^\#(\U)$ over $K$.

One can check that if $a$ is a solution of \eqref{logeq}, then the set of solutions in $\U$ is the coset $a\cdot (G,t)^\#$. Thus, if $L=K\langle a\rangle$ is a Galois extension for the $\alpha$-log equation, one can think of $L\langle (G,t)^\#\rangle$ as a ``splitting field" for \eqref{logeq}. Moreover, it follows from the definitions that $L$ is in fact a generalized strongly normal extension with $X=(G,t)^\#$. It also follows from \cite[\S3]{LSPillay} that the Galois group 
\begin{equation}\label{Galgroup}
\Aut_\Pi(L\langle (G,t)^\#\rangle/K\langle (G,t)^\#\rangle),
\end{equation}
together with its regular action on $tp(a/K)$, is isomorphic via $\sigma\mapsto a^{-1}\cdot\sigma(a)$ to $H^\#$ for some algebraic D-subgroup $H$ of $(G,t)$ over $K$, and there is a natural Galois correspondence between intermediate differential fields and algebraic $D$-subgroups of $H$ defined over $K$.

Differential Galois extensions of log-differential equations (as above) are examples of $(G,t)^\#$-strongly normal extensions where we may assume that $(G,t)^\#$ is the Galois group \eqref{Galgroup} of the extension. One can now ask whether all such normal extensions arise this way. In the current set up, there are a few reasons why this cannot be the case. For instance, the Galois extensions of equations such as \eqref{logeq} have finite transcendence degree over $K$ (since $(G,t)^\#$ is finite-dimensional); whereas there are generalized strongly normal extensions of infinite transcedence degree as long as $m\geq 2$ (i.e., there are at least two derivations), see \cite{CaSi,Landesman} or \cite[Example~6.3]{LS1}.

Galois extensions of logarithmic differential equations of infinite transcendence appear in \cite{CaSi,Landesman} and more generally in \cite{LS1}. Their construction requires a parameterized version of D-groups. One fixes a partition of the set of derivations $\Pi=\DD\cup\D$, with $\DD$ nonempty, and works relative to the theory of $\D$-fields; that is, one replaces algebraic geometry with $\D$-algebraic geometry. A parameterized D-group w.r.t. $\DD/\D$ is then a $\D$-algebraic group $G$ with a $\D$-regular section $t$ of the parameterized prolongation $\tau_{\DD/\D}G$ (further details are given in Section \ref{sec:defDvar}). Much of what we explained above translates to the parameterized setup. However, now the set of sharp points $(G,t)^\#$ is defined by the equation
$$\nabla_\DD(x)=t(x)$$
where $\nabla_\DD(x)=(x,D_1x,\dots,D_rx)$ with $\DD=\{D_1,\dots,D_r\}$, and hence it is in general an infinite-dimensional $\Pi$-group (in other words, of positive $\Pi$-type). Similarly, for $\alpha$ a $K$-point of $\tau_{\DD/\D}G_e$, a differential Galois extension $L/K$ for the $\alpha$-log differential equation $\nabla_\DD(x)=\alpha\cdot t(x)$ is generally of infinite transcendence degree over $K$. These Galois extensions are examples of $(G,t)^\#$-strongly normal extensions (where we may assume that $(G,t)^\#$ is the Galois group of the extension) and were studied in \cite{LS1}. One can now reformulate the question above to: Do all generalized strongly normal extensions arise as differential Galois extensions for log-differential equations on their (parameterized) Galois group $(G,t)$? The close connection between differential-algebraic torsors and differential Galois extensions tells us that the answer is still \emph{no}.

%Turns out the answer is still \emph{no}, and the reason for this is the strong connection between differential-algebraic torsors and differential Galois extensions.

%in particular, the parameterized logarithmic derivative $\ell_s$ is defined as 
%$$\ell(g)=(g,\d_{x_1}g,\cdots,\d_{x_r}g)\cdot (s(g))^{-1}$$
%and its kernel $G^\#$ is not necessarily finite dimensional (in fact its $\D$-type equals the $\D_t$-type of $G$). Differential Galois extensions generated by solutions to parameterized logarithmic equations are again examples of generalized strongly normal extensions.

%I need to mention that Anand approach was finite dimensional; but we need to consider parameterized objects. 
%{\bf Kolchin does not seem to have an example in DAAG (1973) showing that in the theorem one needs D-torsors.}

%However, even allowing solutions to logarithmic differential equations on parametric D-groups, may fail to generate all generalized strongly normal extensions over a given differential field. This is due to the close connection between differential algebraic torsors and differential Galois extensions.

Briefly, a parameterized D-torsor $(V,s)$ for $(G,t)$ is (right) $\D$-torsor $V$ for $G$ such that $s$ is compatible with the action (see Section~\ref{sec:defDvar} for precise definitions). For notation sake, we sometimes say that $(V,s)\circlearrowright (G,t)$ is a parameterized $\DD/\D$-torsor. The $\#$-differential equation of $(V,s)$ is simply $\nabla_\DD(x)=s(x)$. As we suggested above, our main result (Theorem~\ref{thm:mainresult}) will say that (up to a $GL_m(\mathbb Q)$-transformation of $\Pi$) any generalized strongly normal extension is the Galois extension of the $\#$-differential equation on a parameterized D-torsor $(V,s)$ for the Galois group $(G,t)$ of the extension. Furthermore, we show that such extension is the Galois extension of a log-differential equation on $(G,t)$ if and only if $V$ is a trivial torsor (i.e., $V$ has a $K$-point).

\begin{remark}\label{rem:example}
    The above suggests that the existence of nontrivial torsors implies the existence of generalized strongly normal extensions which are \emph{not} differential Galois extension for log-equations on their Galois group. This is in fact the case and can be seen even in the one derivation (ordinary) case using PV-extensions. Indeed, in \cite{JuanLedet}, they provide, for each $n\geq 3$, a PV-extension with Galois group $SO_n$ which is the function field of a nontrivial $SO_n$-torsor \footnote{We are thankful to Michael Singer for pointing out this example.}.
    \end{remark}

It is not clear to us whether or not Kolchin was aware of an example of the phenomenon in the above remark. However, the last theorem in Kolchin's book~\cite{Kolchin1} gives an associated positive result. This result characterizes strongly normal extensions, in our terminology these are precisely $\U^\Pi$-strongly normal extension with $\U^\Pi$ the field of $\Pi$-constants of $\U$. Kolchin's characterization, as proved in \cite{Kolchin1}, uses what he there calls Differential Galois Cohomology, which he later calls Kovacic Cohomology in \cite{Kolchin2}. We now restate Kolchin's theorem in our language and notation as it follows in this form from our main theorem.

%Despite the (negative) example in the above remark, the last theorem in Kolchin's book \cite{Kolchin1} gives a positive result. This result characterizes strongly normal extensions, in our terminology these are precisely $\U^\Pi$-strongly normal extension with $\U^\Pi$ the field of $\Pi$-constants of $\U$. 
%i.e., strongly normal extensions with Galois group embedding in a fixed algebraic group $G$ defined over the constants $C_K$. It will be a special case of our main theorem.
%Kolchin's characterization has two parts. Firstly, a point on a torsor for such a group $G$ generates a $G$-strongly normal extension exactly when it introduces no new constants. Secondly, all such $G$-extensions can be generated by a solution to a logarithmic differential equation on a torsor for $G$. 
%Kolchin's characterization, as proved in \cite{Kolchin1}, uses what he there calls Differential Galois Cohomology, but he then calls this Kovacic Cohomology in \cite{Kolchin2}.  %In the last section we will comment on Kolchin's proof and it's relation to ours here. We now restate Kolchin's theorem in our language and notation as it follows in this form from our main theorem.

\begin{theorem}\cite[Theorem 9, Chapter 6]{Kolchin1}\label{them:Vprim} Let $\U^\Pi$ and $K^\Pi$ denote the $\Pi$-constants of $\U$ and $K$, respectively.
    %Let $K$ be a differential field with constant field $C_K$. Let $G$ be a connected algebraic group defined over $C_K$. Then 
    \begin{enumerate}[a)]
        \item Let $V$ be an algebraic (right) torsor over $K$ for a connected algebraic group $G$ over $K^\Pi$. A solution $a$ to a logarithmic differential equation on $V$ over $K$ generates a strongly normal extension $L=K\langle a \rangle$ of $K$ precisely when $L^\Pi=K^\Pi$ (equivalently $L \cap \U^\Pi=K^\Pi$). 
%\textcolor{red}{Are $K^{\Pi}$ and $C_K$ the same in the statement here? I think we should use one of the two conventions.}
        \item Let $L$ be a strongly normal extension of $K$ with Galois group $G(\U^\Pi)$ with $G$ a connected algebraic group over $K^\Pi$. Then, there is an $a\in L$ generating $L$ over $K$ such that $a$ is the solution to a logarithmic differential equation on an algebraic torsor $V$ over $K$ for $G$. Such an $a$ is called a $V$-primitive element by Kolchin. Moreover, if $\HH^1(K,G)$ is trivial, one can then take $V=G$ and $a$ is called a $G$-primitive.
        %and call $a$ a $G$-primitive.
    \end{enumerate}
\end{theorem}

%\begin{remark}
 %   In the linear (Picard-Vessiot) case even more can be said, namely, fixing a linear algebraic group $G$ over the constants, should there exist one strongly normal (by the linearity of $G$, it will be Picard-Vessiot) extension with Galois group $G$, then the set of $G$-strongly normal extensions is exactly in bijection with $\HH^1_{\text{alg}}(C_K,G)$. Putting this into a general model-theoretic context suitable for the present setting of $DCF_{0,m}$ is part of forthcoming work by the second author.
%\end{remark}

In \cite[Proposition 3.4]{Pillay1} and \cite[Remark 3.8]{Pillay2}, Pillay partially generalized the above theorem of Kolchin to the finite-dimensional generalized strongly normal setting. 
%It is maybe worth mentioning that the category of geometric objects considered in \cite{Pillay1} is the model-theoretic definable category which coincides with the differential algebraic category of Kolchin when working with groups and torsors. 
%In particular, \cite[Proposition 3.4]{Pillay1} includes additional cohomological restrictions. 
Again we quote it slightly adapted to our terminology. 
%The characterization references the model-theoretic version of ``$G$-primitive" elements:

%\begin{proposition}{\cite[Proposition 3.4]{Pillay1}} \
%Let $K$ be a differential field. Let $G$ be a $K$-definable group $K$-definably embedding in $H$, a connected algebraic group defined over $K$ (such an $H$ can be found by results of \cite{Pillay4}). Then 

%\begin{enumerate}[a)]
 %   \item A point $\alpha$ on a $K$-definable left coset of $H/G$ (a specific kind of $K$-definable torsor for $G$) generates a generalized ($G$-)strongly normal extension $L=K\langle \alpha\rangle$ of $K$ exactly when $G(L) = G(K)$.
  %  \item If $K$ is algebraically closed and $L/K$ is a generalized $G$-strongly normal extension, then $L$ is generated by a point on a $K$-definable left coset of $H/G$. 
%\end{enumerate}
 %   A point $\alpha$ on a $K$-definable left coset $H/G$ is called a $(H,G)$-primitive element.
%\end{proposition}

\begin{proposition}{\cite[Proposition 3.4]{Pillay1}} \& \cite[Remark 3.8]{Pillay2}\label{prop:Pillayresult} \
\begin{enumerate}[a)]
    \item Let $(G,t)$ be an algebraic D-group over $K$. A solution $a$ to a log-differential equation on $(G,t)$ generates a $(G,t)^\#$-strongly normal extension $L=K\langle a\rangle$ precisely when $(G,t)^\#(L)=(G,t)^\#(K)$.
    \item Let $L/K$ be a generalized strongly normal extension with Galois group of the form $(G,t)^\#$ where $(G,t)$ is an algebraic D-group over $K$ (this implies in particular that $L$ must be of finite transcendence degree over $K$). Under the additional assumption that $K$ is algebraically closed, we get that $L$ is generated over $K$ by a solution to a log-differential equation on $(G,t)$. Such solution is called a $G$-primitive.
\end{enumerate}
    %A point $\alpha$ on a $K$-definable left coset $H/G$ is called a $(H,G)$-primitive element.
\end{proposition}

%\begin{remark}
 %   We note that the results of \cite{LSPillay} show that the hypothesis $G(K)=G(K^{\diff})$ in part a) of the theorem as it appears in \cite{Pillay1} is not needed. 
%\end{remark}

%In \cite{Pillay2}, the single-derivation generalized strongly normal theory is developed with reference to the category of algebraic $D$-varieties. 
We note that in the introduction of Pillay's \cite{Pillay2} a complete generalization of the above theorem of Kolchin (Theorem~\ref{them:Vprim}) to the setup of generalized strongly normal extensions is mentioned as a possible aim but is not carried out there; namely, changing part a) in Proposition~\ref{prop:Pillayresult} to  allow for arbitrary definable torsors, and changing part b) to drop the assumption that $K$ is algebraically closed. This ``possible aim" is precisely what we achieve in this paper. Let us note that the connection between these two assumptions in parts a) and b) is contained in the following cohomological proposition.
%(due to Kolchin literally part (v) of the corollary to Theorem 4 Section 3 of Chapter VII of \cite{Kolchin2}) which in \cite{Pillay1} appears prior to the previous result (again we change notation for clarity). 
Recall that $\HH^1_{\Pi}$ denotes Kolchin's constrained cohomology, or equivalently, the definable Galois cohomology in $\DCF_{0,m}$ (taken with respect to the $\Pi$-differential closure);
while $\HH^1_{\alg}$ is the usual algebraic Galois cohomology, or equivalently, definable Galois cohomology in the theory $\ACF_0$ (taken with respect to the algebraic closure). 

\begin{proposition}\cite[Proposition 3.2]{Pillay1}\label{prop: alg triv}
Let $G$ be an algebraic group over $K$. If $K$ is algebraically closed, then
$\HH^1_{\Pi}(K, G)$ is trivial.    
\end{proposition}

%An immediate consequence of this cohomological fact is used implicitly in the proof of \cite[Proposition 3.4]{Pillay1} and can be shown more cleanly shown using the long exact sequence in definable Galois cohomology developed by the authors and Pillay in \cite{LSMP} and which we will recall in Section \ref{sec:modPrelims}: 

%\begin{fact}
%If $G$ is a $K$-definable group $K$-definably embedding in $H$, and if $K$ is algebraically closed, then every $K$-definable torsor for $G$ is, up to $K$-definable isomorphism, a $K$-definable left coset of $G$ in $H$. 
%\end{fact}

This proposition can be seen as an immediate consequence of the following theorem of Kolchin, which we refer as \emph{Kolchin's Cohomology Theorem}. 

%and either the model-theoretic or differential-algebraic version of Fact \ref{fact: SES}, the exact sequence in cohomology induced by a normal extension. 

\begin{theorem}\cite[\S VII.3]{Kolchin2}\label{thm: Kolchin's equivalence}
    Let $G$ be an algebraic group over $K$. Then, $$\HH^1_{\Pi}(K, G) \cong \HH^1_{\alg}(K, G)$$
\end{theorem} 

In the first author's paper \cite{LS1}, a generalization of  Proposition~\ref{prop:Pillayresult} to the parameterized setting was established. Recall that, in the parameterized context, we have a partition $\Pi=\DD\cup\D$. We refer to Section~\ref{sec:defDvar} to the notion of a set of derivations $\D$ bounding the $\Pi$-type of a differential field extension.
%We quote directly from \cite{LS1}. The full set of derivations is denoted $\Pi$.   The following is really two theorems from \cite{LS1} which are quote directly. In Section \ref{sec:defDvar} we explain the geometric terminology of parameterized $D$-torsors, $\D$ bounding the $\Pi$-type, etc. 

\begin{theorem}\label{thm:fromOmar}  \cite[Theorems 5.11 and 5.7]{LS1}
\begin{enumerate}[a)]
        \item  Let $(G,t)$ be a parameterized $D$-group w.r.t $\DD/\D$ defined over $K$. A solution $a$ to a log-differential equation of $(G,t)$ generates a $(G,t)^\#$-strongly normal extension precisely when $(G, t)^{\#}(L)=(G, t)^{\#}(K)$.        
        
        \item Let $L/K$ be a generalized strongly normal extension such that $\D$ bounds the $\Pi$-type of $L$ over $K$. Let $(G,t)^\#$ be the Galois group of $L/K$ where $(G,t)$ is a parameterized D-group w.r.t. $\DD/\D$. Under the additional assumption that $(K,\D)$ is $\Delta$-closed, $L$ is generated by a solution to a (parameterized) log-differential equation on $(G,t)$. 
    \end{enumerate}
\end{theorem}

%\begin{remark}
 %   Again in hindsight, the results of \cite{LSPillay} show that the hypothesis from part 1) of the above theorem, $(G, s)^{\sharp}(K) = (G, s)^{\sharp}(\bar{K})$, is not needed. 
%\end{remark}

\begin{remark}
    We note that in the special case when $G=\U$ and $t$ is the zero section, part a) of the above theorem appears in Landesman exploration of $\U^\DD$-strongly normal extensions, see \cite[Proposition 3.64]{Landesman}.
\end{remark}

It is worth noting that the following parameterized analogue of Proposition \ref{prop: alg triv} is used in the proof of Theorem~\ref{thm:fromOmar}, although this is not stated explicitly in~\cite{LS1}.

%It is worth noting that used in the proof of Theorem~\ref{thm:fromOmar}, though not stated explicitly in~\cite{LS1}, is the following parameterized analogue of Proposition \ref{prop: alg triv}.

\begin{proposition}\label{prop: omar's triviality result}
    Let $G$ be $\D$-group defined over $K$. If $(K,\D)$ is $\D$-closed, then $\HH^1_{\Pi}(K, G)$ is trivial.  
\end{proposition}

Although a proof of the above proposition is suggested in \cite{LS1}, it follows immediately from a more general statement, a parameterized analogue of Kolchin's Cohomology Theorem (Theorem~\ref{thm: Kolchin's equivalence}) which to our knowledge does not appear elsewhere and could be considered of interest. We prove the parameterized version in Section~\ref{sec:parKoltheorem}; namely:

%This, in turn, follows immediately from a parameterized analogue of Kolchin\textcolor{red}{'s Cohomology Theorem} cohomology theorem (Theorem~\ref{thm: Kolchin's equivalence}) which to our knowledge does not appear elsewhere and could be of independent interest. We prove the parameterized version in Section~\ref{sec:parKoltheorem}; namely:

\begin{theorem}\label{thm: analogue kolchin}
    Let $G$ be a $\D$-group over $K$. Then, there is a canonical isomorphism
    $$\HH^1_\Pi(K,G)\cong \HH^1_\D(K,G),$$
    where the latter denotes Kolchin's cohomology with respect to $\D$ (equivalently, definable cohomology in the reduct $(\U,\D)$ taking $\D$-diff closures).
\end{theorem}

%This parameterized version of Theorem \ref{thm: Kolchin's equivalence}, gives two consequences, first by Fact \ref{fact: SES}, we obtain Proposition \ref{prop: omar's triviality result}. Secondly, this reconciles the choice of the map in our characterization of when the extension is generated by an equation on the group (``moreover" clause of the main result). In the last section we give the precise relationship between the two maps $\iota$ and $\iota^1$. However, we believe Theorem \ref{thm: analogue kolchin} to be of interest in its own right.

As we already hinted above, one of the main purposes of this paper is to improve the content of Theorem~\ref{thm:fromOmar}. Firstly, in part (a), we consider differential Galois extensions for general parameterized D-torsors (instead of restricting ourselves to log-equations on parameterized D-groups); and secondly, in part (b), we drop the assumption that $(K,\D)$ is $\Delta$-closed and give a criterion for exactly when one recovers Galois extensions for log-differential equations on the Galois group. Thus, our results here give a suitably general formulation of the generalized strongly normal theory \emph{without additional closure or cohomological assumptions}. 

We now state our main theorem, generalizing Kolchin's Theorem~\ref{them:Vprim}, Pillay's Proposition~\ref{prop:Pillayresult}, as well as the first author's Theorem~\ref{thm:fromOmar}.  

\begin{theorem}[{\bf Main result}] \label{thm:mainresult} \
\begin{enumerate}
    \item Let $\DD\cup\D$ be a partition of $\Pi$ with $\DD$ nonempty and $(V,s)\circlearrowright (G,t)$ a parameterized $\DD/\D$-torsor. A solution $a$ to the $\#$-differential equation of $(V,s)$ generates a $(G,t)^\#$-strongly normal extension $L=K\langle a \rangle_\Pi$ if and only if $(G,t)^\#(L)=(G,t)^\#(K)$.
    
    \item Let $L$ be a generalized strongly normal extension %Let $L = K \l b \r_\Pi$ be a generalized $X$-strongly normal extension 
    with $L/K$ regular (so the differential Galois group is connected). Then, after a $GL_m(\mathbb Q)$-transformation of $\Pi$, there exists a partition $\DD\cup\D$ with $\D$ witnessing the $\Pi$-type of $L/K$ and a parameterized $\DD/\D$-torsor $(V,s)\circlearrowright (G,t)$ over $K$ such that $(G,t)^\#$ is the differential Galois group of $L/K$ and $L$ is a differential Galois extension for $(V,s)$. In fact, $L$ is isomorphic to the $\Pi$-function field of $(V,s)^\#$. 
    
    Moreover, $L/K$ is the differential Galois extension for a log-differential equation on $(G,t)$ if and only if the image of the class associated to the $\Pi$-torsor $(V,s)^\#\circlearrowright (G,t)^\#$
    under the natural map in cohomology (using Theorem~\ref{thm: analogue kolchin})
    $$\HH^1_{\Pi}(K,(G,t)^\#) \to \HH^1_{\D}(K,G)$$
    is trivial. In particular, this occurs when $(K,\D)$ is $\D$-closed as then $\HH^1_\D(K,G)$ is trivial.
\end{enumerate}
\end{theorem}

%Part a) will be proven in Section \ref{sec:defDvar}. 

\medskip

In Section~\ref{sec:modPrelims} we provide a model-theoretic version of the above result at the level of generality of definable cohomology in totally transcendental theories. 

%In fact, we will use it in the proof of the main theorem specialised to the theory DCF$_{0,m}$.

\

%\textcolor{red}{Do we want to say here that section 2 contains generalization of the theorem at the level of a totally transcendental theory (in the definable category)? Theorem 2.8 below is now stated less in the form of the main theorem above. It's alright with me if we leave this as is.}

%\textcolor{red}{Other remarks: So I think one thing which was subtle to me for a while but makes sense to me now: The Kolchin-Lang paper shows that every strongly normal extension has a model which is a torsor for the Galois group but there is no mention of an equation on the torsor. I think this perspective is from Bialynicki-Birula but I have to check this. I am fine not mentioning this. }

%{\color{blue}OLS: I put a comment at the end of Section 2 to address that we obtain a general model-theoretic version.

%I would suggest not mentioning the Bialynicki-Birula approach. I feel it wouldn't add much and might be cumbersome to precisely state it, and readers familiar with it will probably see how to do the translation. Note that Anand's mentions this approach in his DGT paper as Prop 2.15}

\section{Some model-theoretic generalities}\label{sec:modPrelims}

In this section we lay out the general conventions and results that will be used in the rest of the paper. To start with, we review and elaborate on the results from~\cite{LSPillay} around model-theoretic definable Galois theory. We fix a first-order totally transcendental theory $T$ with elimination of imaginaries. Let $\U$ be a monster model of $T$ saturated in a sufficiently large cardinal $\kappa$. Let $A$ be a small (with respect to $\kappa$) $\dcl$-closed set of parameters over which we will work. Let $M$ be a copy of the prime model over $A$. 

Let $X$ be an $A$-definable set and $q \in S_n(A)$. Following \cite{LSPillay} we say that $q$ is strongly internal to $X$ if for any pair of realizations $b_1$ and $b_2$ of $q$, $b_2 \in \dcl(b_1, A, X)$. Furthermore, recall that $q$ is said to be weakly orthogonal to $X$ if $q$ has a unique extension to a type over $A\cup X$. 

\begin{fact}\cite[Lemma 2.2(i)]{LSPillay}\label{weak orth fact}
    The type $q$ is weakly orthogonal to $X$ if and only if for any (equivalently some) realization $b$ of $q$ we have $X(\dcl(A,b))=X(A)$.
\end{fact}
%For convenience of presentation, for the rest of this section we work with $A$ as a set of parameters added to our language (in particular, $A\subseteq \dcl(\emptyset)$). 
We say that a $\dcl$-closed set $B$ is a Galois extension of $A$ relative to $X$ if $B=\dcl(A,b)$ and $q=tp(b/A)$ is strongly internal and weakly orthogonal to $X$. These conditions capture the idea that Galois extensions of $A$ have a (extrinsic) definable Galois group living in $\dcl(A,X)$. We make this precise in the following fact. We refer to part (2) of this fact as \emph{The Torsor Theorem} and we will use it heavily later. Part (3) is sometimes referred to as the indirect Galois correspondence. 

\begin{fact}\cite[\S 2]{LSPillay}\label{fact: torsor theorem} Let $X$ be an $A$-definable set and $q=tp(b/A)$ be a type which is both strongly internal and weakly orthogonal to $X$. Let $Q$ denote set of realizations of $q$ in $\U$ and $Aut(Q/X,A)$ denote the group of permutations of $Q$ which are elementary over $A\cup X$.
    \begin{enumerate}
        \item The type $q$ is isolated and hence has a realization $b$ in $M$. 
        \item $Aut(Q/X,A)$ is isomorphic to an $A$-definable group $H$ living in $\dcl(A,X)$. $H$ is connected if and only if $\acl(A)\cap\dcl(A,b)=A$. Furthermore, $Q$ is an $A$-definable right torsor for $H$.
        \item The group $H$ is called the extrinsic Galois group of the Galois extension extension $B=\dcl(A,b)$. Moreover, $B$ is a Galois extension of $A$ relative to $H$, and there is a Galois correspondence between intermediate $\dcl$-closed sets of $B/A$ and $A$-definable subgroups of $H$.
    \end{enumerate}
\end{fact} 

\begin{example} Let $T=\DCF_{0,m}$. Denote the derivations by $\Pi$ and the field of $\Pi$-constants by $\U^\Pi$. Further details of the following examples appear in \cite[\S3]{LSPillay}.
    \begin{enumerate}
        \item If $L=K\langle b\rangle$ is a strongly normal extension in the sense of Kolchin \cite[Chapt VI]{Kolchin1}, then it is a Galois extension relative to $\U^\Pi$ in the above sense; namely, $tp(b/K)$ is strongly internal and weakly orthogonal to $\U^\Pi$. In this case the extrinsic differential Galois group can be identified with an algebraic group living in $\U^\Pi$.
        \item If $L/K$ is a (parameterized) $\Delta$-strongly normal extension in the sense of Landesman \cite{Landesman} where $\Delta\subset \Pi$, then it is a Galois extension relative $\U^{\DD}$ where $\DD=\Pi\setminus \D$, and in this case the differential Galois group is a $\D$-algebraic group in $\U^{\DD}$.
        \item More generally, if $L/K$ is an $X$-strongly normal extension in the sense of~\cite{LS1} where $X$ is a $K$-definable set, then it is Galois extension relative to $X$ with Galois group isomorphic to a $\Pi$-algebraic group.
        \end{enumerate}
\end{example}

%\begin{remark}\label{rem: strong normality internality rephrasing}
  %  Set $T = \DCF_{0,m}$. Let $L = K\langle \beta\rangle/K$ be a generalized $X$-strongly normal extension. The strong normality condition tells us exactly that $\tp(\beta/K)$ is strongly internal to $X$ and the no new constants condition tells us exactly that $\tp(\beta/K)$ is weakly orthogonal to $X$. In this case the extrinsic differential Galois group will be a parametric group living in $\dcl^{\eq}(X)$. 
%\end{remark}

We now recall some fundamental results on definable Galois cohomology from~\cite{LSMP} and \cite{Pillay3}.
%\textit{Remarks on Galois cohomology and definability} 
%and 
%\textit{More on Galois cohomology, definability, and differential algebraic groups} \cite{LSMP}. 
%We also give a general result from the second author's PhD thesis~ 
%\textit{Differential Field Arithmetic} 
%\cite{Me2} there stated for the theory $\DCF_0$ but which works in the %present totally transcendental setting.
%We continue in the same setting as above, with $T$ a totally transcendental theory with elimination of imaginaries, $\bar{M}$ a monster model, $A$ a small set of parameters and $M$ a copy of the prime model over $A$. 
Let $G$ be an $A$-definable group. 
%Let $P$ be an $A$-definable right torsor for $G$. 
%We call $P$ a (right $A$-)definable homogeneous space if the action is transitive. We call $P$ a (right $A$-)definable principal homogeneous space, or PHS, or torsor, if the action is strictly transitive, meaning that for any $b_1,b_2 \in P$, there is a unique $g \in G$ with $b_1g = b_2$. 
%A pair $(P,p)$ consisting of a PHS $P$ together with a specified point $p \in P$ is called a pointed PHS (for the group $G$).  
An $A$-definable isomorphism between $A$-definable right torsors for $G$ is an $A$-definable bijection that commutes with the right action.
%An isomorphism of pointed PHSs is an isomorphism of the PHSs preserving the basepoint.
%\begin{remark}
   % When considering definable Galois cohomology in a totally transcendental theory, one works with reference to a fixed copy, $M$, of the prime model. Practically, this is to make use of the atomicity and homogeneity properties of the prime model. If $P$ is a definable right PHS for $G$, then $P(M)$ is a right PHS for $G(M)$ abstractly. By a pointed PHS, $(P(M),p)$, we mean the $M$-points of a definable (in the sense of $T$ or $\bar{M}$) PHS $P$ together with a specified $p \in P(M)$.
%\end{remark}
We denote the set of $A$-definable isomorphism classes 
%(in $M$) 
of $A$-definable right torsors for $G$ by $\mathcal{P}_{\defin}(A,G)$.

We also recall the notion of definable cocycles from \cite{Pillay3}. We call a function $\varphi:\Aut(M/A) \to G(M)$ an $A$-definable cocycle if the following two conditions are satisfied:

\begin{enumerate}
    \item $\forall \sigma, \tau \in \Aut(M/A)$, $\varphi(\sigma\tau) = \varphi(\sigma)\sigma(\varphi(\tau))$
    \item there is a tuple $a$ from $M$ and an $A$-definable function $h(x,y)$ such that $h(a,\sigma(a)) = \varphi(\sigma)$ for all $\sigma \in \Aut(M/A)$.
\end{enumerate}

Furthermore, two cocycles $\varphi$ and $\psi$ are said to be cohomologous if there is $g \in G(M)$ such that for all $\sigma \in \Aut(M/A)$ we have $\varphi(\sigma) = g^{-1}\psi(\sigma)\sigma(g)$. Cohomology is an equivalence relation on cocycles. The constant cocycle takes all of $\Aut(M/A)$ to the identity of $G(M)$. We call a definable cocycle trivial if it is cohomologous to the constant cocycle.

We denote the set of definable cocycles from $\Aut(M/A)$ to $G(M)$ by 
%either $\ZZ^1_{\defin}(M/A,G(M))$ or 
$\ZZ^1_{\defin}(A,G)$.
%with the constant cocycle as the basepoint. 
We denote the quotient of $\ZZ^1_{\defin}(A,G)$ by the cohomology relation by 
%$\HH^1_{\defin}(M/A,G(M))$ or, more compactly, 
$\HH^1_{\defin}(A,G)$, and call it the first (nonabelian) definable Galois cohomology set.

\begin{fact}\cite{Pillay3}
    There is a natural isomorphism 
    %of pointed sets 
    %$$\mathcal{P}_{\defin,*}(M/A,G(M)) \cong \ZZ^1_{\defin}(M/A,G(M))$$
    %which induces another natural isomorphism of pointed sets 
    $$\mathcal{P}_{\defin}(A,G) \cong \HH^1_{\defin}(A,G).$$ %by forgetting the basepoint on the left and taking cohomology on the right.
\end{fact}

\

We now recall the long exact sequence in definable cohomology from 
%\textit{More on Galois cohomology, definability and differential algebraic groups} 
\cite{LSMP}.
%, not requiring a normal subgroup. 
This exact sequence is a primary tool for computing definable Galois cohomology. But first let us recall that given an $A$-definable subgroup $H\leq G$, with $H$ not necessarily normal in $G$, by elimination of imaginaries the pointed set of left cosets of $H$ in $G$, $G/H$, can be identified with an $A$-definable set and the quotient map $G\to G/H$ is $A$-definable.

\begin{fact}\cite{LSMP}\label{fact: LES}
    Let $H$ be an $A$-definable subgroup of $G$. This induces an exact sequence in definable Galois cohomology:

    \vspace{-4mm}

    $$1 \to H(A) \to G(A) \to (G/H)(A) \xrightarrow{\delta^1} \HH^1_{\defin}(A,H) \xrightarrow{\iota^1} \HH^1_{\defin}(A,G)$$
\end{fact}

\begin{remark}\label{rmk: subtorsor} We make a few comments on the maps $\delta^1$ and $\iota^1$. See \cite{LSMP}, \cite{Me}, \cite{Me2}, and \cite{Pillay3} for more details.
    \begin{enumerate}
        \item The connecting map $\delta^1$ sends an $A$-definable coset $gH$ of $H$ in $G$ to the cohomology class associated to $gH$ considered as a right $A$-definable torsor for $H$.
        \item As $H(M) \subseteq G(M)$, a definable cocycle valued in $H(M)$ can also be seen as valued in $G(M)$ and this identification induces $\iota^1$ at the level of cocycles. 
        \item Let $V$ be an $A$-definable right torsor for $H$. Let $a \in V(M)$ and $h(x,y)$ be the associated $A$-definable function satisfying $a\cdot h(a,\sigma(a)) = \sigma(a)$ for any $\sigma \in \Aut(M/A)$. Let $Z$ be the set of realizations of the isolated type $\tp(a/A)$. Under $\iota^1$, the class of $V$ is sent to the class of $Z\times G/ \sim$ where $(a_1,g_1) \sim (a_2,g_2)$ iff $h(a_1,a_2)g_2 = g_1$ which is an $A$-definable right torsor for $G$ (under then natural action $(a,g_1)\cdot g=(a_1,g_1\, g)$).
        \item Let $V\circlearrowright H$ be an $A$-definable sub-torsor for $W\circlearrowright G$. That is, $V \subseteq W$ and $H \leq G$ and the action of $H$ on $V$ is the restriction of the action of $G$ on $W$. Then, under $\iota^1$, the class of $V$ is sent to that of $W$. Indeed, by~(3), the class of $V$ is sent to $Z\times G/ \sim$ where $Z$ is the set of realizations of the isolated type $\tp(a/A)$ for $a \in V(M)$. Now, using that $a\in W$, consider the map $\rho:Z\times G\to W$ given simply by $(a,g)\mapsto a\cdot g$. Clearly $\rho$ is onto. One can check that $\rho$ preserves the equivalence relation $\sim$ and hence induces an $A$-isomorphism between $Z\times G/\sim$ and $W$. That this construction recovers $W$ is exactly the proof of the cocycle-torsor correspondence originally from~\cite{Pillay3}. 
    \end{enumerate}
\end{remark}

Now we include some material from the second author's PhD thesis \cite{Me2}, which appears there specialized to $\DCF_0$ but which works at the present level of generality. The first part of the following definition is due to Pillay and Sokolovic from 
%their paper \textit{Superstable differential fields} 
\cite{PSokolovic}. This definition is a model-theoretic generalization of Kolchin's $G$- and $V$-primitive elements, respectively.

\begin{definition}
    Let $H$ be an $A$-definable subgroup of the $A$-definable group $G$. 
    \begin{enumerate}
        \item \cite{PSokolovic} Let $\pi:G \to G/H$ be the $A$-definable quotient map. An element $a\in G$ is said to be a $(G,H)$-primitive over $A$ if $\pi(a) \in (G/H)(A)$; namely, $\pi(a)$ is an $A$-point.
        \item \cite{Me2} Let $W$ be an $A$-definable right torsor for $G$. As $H$ is a subgroup of $G$, $H$ acts on $W$ and the orbits of the action are the classes of an $A$-definable equivalence relation. Let $\pi:W \to W/H$ be the $A$-definable quotient map (given by elimination of imaginaries). An element $a\in W$ is said to be a $(W,G,H)$-primitive over $A$ if $\pi(a) \in (W/H)(A)$.
    \end{enumerate}
\end{definition}

We conclude this section with the following theorem which gives the relationship between $A$-definable torsors and model-theoretic primitives (as in the Definition above). Part of the proof is adapted from \cite[\S 3.3]{Me2}. 

%\begin{lemma}\label{lem: mod prim}
%    Let $X$ be a a right $A$-definable PHS for a $A$-definable group $H$. Let $H$ embed via a $A$-definable homomorphism $\iota$ in a $A$-definable group $G$.
    
%    If the definable Galois cohomology of $G$ is trivial,  i.e. $X$ is $A$-definably isomorphic to an $A$-definable left coset of $H$ in $G$.
    
%    Otherwise, there is some nontrivial PHS $W$ for $G$ such that $X$ is $A$-definably isomorphic to the preimage $\pi^{-1}(\pi(b)) \subseteq W(M)$ for some $(W,G,H)$-primitive element $b \in W(M)$ such that $\pi(\bar{b}) \in (W/H)(A)$, i.e $X$ is $A$-definably isomorphic to an orbit of $H$ acting on $W$.
%\end{lemma}

%\begin{proof}
    
%\end{proof}

\begin{theorem}\label{thm: cohomological fact} Let $\U$ be a monster model of a totally transcendental theory $T$ that eliminates imaginaries, and let $A$ be a (small) set of parameters.
\begin{enumerate}
    \item Let $Q$ be an $A$-definable right torsor for an $A$-definable group $H$, and let $b \in Q(\U)$. Then, $B=\dcl(A,b)$ is a Galois extension of $A$ relative to $H$ (in the sense above) if and only if $H(B)= H(A)$.
    
    \item Let $B/A$ be a Galois extension (in the sense above); namely, $B=\dcl(A,b)$ with $q=tp(b/A)$ strongly internal and weakly orthogonal to some $A$-definable set. Let $Q \circlearrowright H$ be the associated $A$-definable right torsor given by the torsor theorem (i.e., part (ii) of Fact~\ref{fact: torsor theorem}). Suppose $H$ is subgroup of an $A$-definable group $G$. Then, there is an $A$-definable right torsor $W$ for $G$ such that $Q$ is $A$-definably isomorphic to an orbit of $H$ on $W$ and the Galois extension $B/A$ can be generated by a $(W,G,H)$-primitive over $A$.
        
        Moreover, the image of the class associated to the torsor $Q$ under the natural map 
        $$\HH^1_{\defin}(A,H) \xrightarrow{\iota^1} \HH^1_{\defin}(A,G)$$ is trivial if and only if $Q$ is $A$-definably isomorphic to an $A$-definable left coset of $H$ in $G$. And, in this case, the Galois extension $B/A$ can be generated by a $(G,H)$-primitive. This happens in particular when $\HH^1_{\defin}(A,G) = 1$.
\end{enumerate}
\end{theorem}

\begin{proof}
(1)  Since the action of $H$ on $Q$ is transitive we have that: for any $b_1,b_2\in Q(\U)$, there is $h\in H$ with $b_2=b_1\cdot h$, yielding that $q=tp(b/A)$ is strongly internal to $H$. That $q$ is weakly orthogonal to $X$ is just Fact \ref{weak orth fact}. We note that $H$ will not necessarily be the Galois group of $B$.

    %We may replace $H$ in notation and let $(Q,H)$ be the associated $A$-definable PHS given by the torsor theorem, Fact \ref{fact: torsor theorem}. 
    
    %We apply the long exact sequence in definable Galois cohomology induced by the short exact sequence of pointed sets $1 \to H \to G \to G/H \to 1$, Fact \ref{fact: LES}. Note $H$ is not necessarily a normal subgroup. 
    (2) By Fact~\ref{fact: LES}, we have the exact sequence in cohomology:
     \begin{align*}
        1  \to H(A) &\to G(A) \to (G/H)(A) \xrightarrow{\delta^1} 
         \HH^1_{\defin}(A,H) \xrightarrow{\iota^1} \HH^1_{\defin}(A,G).
    \end{align*}
    %where for $d \in (G/H)(A)$, the map $\delta^1(d) := \pi^{-1}(d) \subseteq G$. 
    Let $W$ an $A$-definable (right) torsor for $G$ representing the image under $\iota^1$ of the class of $Q$. By part (3) of Remark~\ref{rmk: subtorsor}, we may assume that $W=Q\times G/\sim$ where $(b_1,g_1)\sim(b_2,g_2)$ iff $h(b_1,b_2)g_2=g_1$. Let $[(b,e)]$ denote the class of $(b,e)\in Z\times G$ under $\sim$ where $e$ is the identity of $G$. Note that $H$ acts on $W$ by $[(b_1,g_1)]\cdot h=[(b_1,g_1\, h)]$. We then denote the orbit of $[(b_1,g_1)]$ under $H$ by $[(b_1,g_1)]\cdot H$. 

    %Let $\iota^1$ sends the class of $Q$ to a class in $\HH^1_{\defin}(A,G)$. Let $W$ be a PHS representing the image of the class of $Q$ under $\iota^1$. 

    %Recall how $\iota^1$ is defined from the PHS perspective. Namely we construct a PHS for $G$ as follows: Let $p \in Q(M)$. Let $P(M)$ be the realizations of $\tp(p/A)$ in $M$, noting that $M$ is atomic. The PHS is $W:= (P(M) \times G(M))/E$ where $E$ is the $A$-definable equivalence relation defined by $(p_1,g_1)E(p_2,g_2)$ if $\exists h \in H$  such that $p_1h = p_2$ and $hg_2 = g_1$.

\medskip

    \noindent {\bf Claim.} $Q$ is $A$-definably isomorphic to the orbit $[(b,e)]\cdot H \subseteq W$. 

\begin{proof}[Proof of Claim]
    First we note that $[(b,e)]\cdot H$ is $\Aut(\U/A)$-invariant and so $A$-definable. Indeed, let $[(b,h)] \in [(b,e)]\cdot H$ and let $\sigma \in \Aut(\U/A)$, then 
    $$\sigma([(b,h)]) = [(\sigma(b),\sigma(h))] = [(bh(b,\sigma(b)),\sigma(h))] = [(b,h(b,\sigma(b))\sigma(h))]\in [(b,e)]\cdot H,$$ 
    where the last equality uses the definition of the equivalence relation $\sim$. Note that in fact this shows that $[(b,e)]\cdot H=[(b_1,e)]\cdot H$ for any $b_1\in Q$.

    For any $x \in Q$, there is a unique $h\in H$ with $x = bh$. Define the map $Q\to [(b,e)]\cdot H$ by $x \mapsto [(b,h)]$. It is straightforward to check that this an $A$-definable bijection.
    
    %This maps has inverse $[(a,h)] \mapsto ah$. It is clear that the inverse is well defined and $A$-definable. The map $x = ph \mapsto [(p,h)]$ apriori is definable over $p$ but in fact it is $\Aut(M/A)$-invariant so $A$-definable.
    %This map is $A$-definable and bijective. 
    \end{proof}
    
    %So $Q$ is $A$-definably isomorphic to an $A$-definable orbit of $H$ on $W(M) := (P(M) \times G(M))/E$. That is, 
    
    It follows that $Q$ is $A$-definably isomorphic to the preimage $\pi^{-1}(\pi([(b,e)])) \subseteq W$ where $\pi:W\to W/H$ is the canonical $A$-definable quotient map. Hence $\pi([(b,e)])$ is an $A$-point of $W/H$; namely, $[(b,e)]$ is a $(W,G,H)$-primitive over $A$. Furthermore,  as $b$ and $[(b,e)]$ are $A$-interdefinable, we get that $B=\dcl(A,[(b,e)])$, as desired.

    For the moreover clause, note that the class of $Q$ is in the kernel of $\iota^1$ if and only if the class of $Q$ is in the image of $\delta^1$, in which case, by part (1) of Remark~\ref{rmk: subtorsor}, $Q$ is $A$-definably isomorphic to an $A$-definable left coset of $H$ in $G$. Let $c\in G$ be the image of $b$ under this $A$-definable isomorphism. In this case, $Q$ is $A$-definably isomorphic to the preimage $\pi^{-1}(\pi(c)) \subseteq G$, where now $\pi:G\to G/H$, and hence $c$ a $(G,H)$-primitive over $A$. Clearly, $b$ and $c$ are $A$-interdefinable.

    %Note of course that if $\HH^1(M,G) = 1$, then any PHS for $H(M)$ is in the image of the connecting map and must be isomorphic to a $A$-definable coset of $H$ in $G$.
\end{proof}

%\textcolor{red}{Do we want to put a comment saying that the above theorem can be seen as a model-theoretic generalization of the main result?}

%{\color{blue}OLS: how about the following: }

The above theorem is a model-theoretic version of our main result (Theorem~\ref{thm:mainresult}) at the level of generality of definable cohomology in totally transcendental theories. We will use it in the proof of the main theorem specialised to the theory DCF$_{0,m}$ where we give a stronger condition in part (2) using the parameterized version of Kolchin's cohomology theorem (see Theorem~\ref{thm: param kolchin}).

%we note however that in part (b) we use Kolchin's parameterize version to give a condition in terms of definable cohomology of the related theory

\

\section{Parameterized D-torsors and their differential Galois extensions}\label{sec:defDvar}

We now work in a sufficiently saturated model $(\U, \Pi)$ of $\DCF_{0,m}$ and we fix a ground $\Pi$-subfield $K$. All $\Pi$-varieties are identified with their set of $\U$-points. Let us start by recalling Kolchin's notions of $\Pi$-type and typical $\Pi$-dimension. Let $a$ be an $n$-tuple from $\U$. From \cite[\S II.12]{Kolchin1}, there exists a numerical polynomial $\omega_{a/K}$, called the Kolchin polynomial of $a$ over $K$ with the following properties:

\begin{itemize}
\item [(i)] For sufficiently large $h\in \mathbb N$, $\omega_{a/K}(h)$ equals the transcendence degree of
$$K(\theta a_i:i\leq n, \theta\in \Theta_\Pi \text{ with } \operatorname{ord}\theta\leq h)$$
over $K$. Here $\Theta_\Pi$ denotes the set of $\Pi$-derivatives.
\item [(ii)] $\deg \omega_{a/K}\leq |\Pi|$.
\item [(iii)] Writing $\omega_{a/K}(x)=\sum_{i}d_i\binom{x+i}{i}$ where $d_i\in \mathbb Z$, then $d_{|\Pi|}$ equals the $\Pi$-transcendence degree of $K\langle a\rangle_\Pi$ over $K$.
\item [(iv)] If $b$ is a (finite) tuple from $K\langle a\rangle_\Pi$, then there exists $h_0\in \mathbb N$ such that for sufficiently large $h\in \mathbb N$ we have $\omega_{b/K}(h)\leq \omega_{a/K}(h+h_0)$.
\end{itemize}

From (iv) we have that the degree of $\omega_{a/K}$ is a $\Pi$-birational invariant and we call this degree the $\Pi$-type of $a$ over $K$, denoted by $\Pi$-type$(a/K)$. Furthermore, again from (iv), we have that writing $\omega_{a/K}(x)=\sum_{i}^{\tau}d_i\binom{x+i}{i}$ where $d_i\in \mathbb Z$ and $\tau=\Pi$-type$(a/K)$, the coefficient $d_\tau$ is also a $\Pi$-birational invariant which is called the typical $\Pi$-dimension of $a$ over $K$ and denoted $\Pi$-dim$(a/K)$.

For $\D\subseteq \Pi$, one says that $\D$ \emph{bounds} the $\Pi$-type$(a/K)$ if the $\Pi$-field $K\langle a\rangle_\Pi$ is finitely $\D$-generated over $K$ (i.e., there is a finite tuple $\alpha$ such that $K\langle a\rangle_\Pi=K\langle \alpha\rangle_\D$). If in addition $|\D|=\Pi$-type$(a/K)$, we say that $\D$ \emph{witnesses} the $\Pi$-type$(a/K)$. These notions were introduced in \cite[\S1]{LS1}.  In his first book, Kolchin proved the following fundamental result on finitely generated differential extensions: it essentially says that one can always find some $\D$ witnessing the $\Pi$-type.

\begin{fact}\label{Koltype}\cite[\S II.13]{Kolchin1}
    Let $a$ be a (finite) tuple from $\U$. Then, after a $GL_m(\mathbb Q)$-transformation of $\Pi$, there exists $\D\subseteq \Pi$ witnessing the $\Pi$-type$(a/K)$. Namely, $|\D|=\Pi$-type$(a/K)$ and a there exists a (finite) tuple $\alpha$ from $\U$ such that
    $$K\langle a\rangle_\Pi=K\langle \alpha\rangle_\D.$$
    %and $\Pi$-type$(a/K)=\D$-type$(\alpha/K)$ and $\Pi$-dim$(a/K)=\D$-dim$(\alpha/K)$.
Furthermore, the $\D$-transcendence degree of $K\langle \alpha\rangle_\D$ over $K$ is $\Pi$-dim$(a/K)$. 
\end{fact}

We now consider a (disjoint) partition $\Pi=\DD\cup\D$ with $\DD=\{D_1,\dots,D_r\}$ nonempty and $\D=\{\d_1,\dots,\d_\ell\}$ (we will think of $\D$ as the parametric derivations). We will denote by $K^\D$ the $\D$-constants of $K$ and by $\Theta_\D$ the set of $\D$-derivatives (namely, operators of the form $\d_1^{e_1}\cdots\d_\ell^{e_\ell}$ for $e_i\geq 0$). 

Let us recall some of the results from \cite{LS1} on parameterized D-varieties with respect to the partition $\DD/\D$ (note that in that paper the terminology ``relative" D-variety was used, but here we use the more common term ``parameterized"). From~\cite[\S3]{LS1}, there is a functor $\tau_{\DD/\D}$ from the category of $\D$-varieties over $K$ to itself with the following three properties:

\begin{enumerate}
\item If $V$ is a $\D$-variety defined over $K^\D$, then $\tau_{\DD/\D}V$ is equal to Kolchin's $\D$-tangent bundle $T_\D V$, see \cite[\S VIII.2]{Kolchin2}. %In general $\tau_{\DD/\D} V$ is a $T_\D V$-torsor.
\item $\tau_{\DD/\D}$ commutes with products; namely, $\tau_{\DD/\D}(V\times W)$ is naturally isomorphic to $\tau_{\DD/\D}V\times \tau_{\DD/\D}W$.
\item For affine $V\subseteq \U^n$, the $\DD$-morphism $\nabla_\DD:\U^n\to \U^{n(r+1)}$ given by 
$$\nabla_\DD(x):= (x,D_1(x),\dots,D_r(x))$$
restricts to $\nabla^V_{\DD}:V \to \tau_{\DD/\D}V$. Moreover, if $f:V\to W$ is a $\D$-morphism defined over $K$, then $\nabla^W_{\DD}\circ f=\tau_{\DD/\D}f\circ \nabla^V_{\DD}$.
\end{enumerate}

In an affine chart of a $\D$-variety $V$ over $K$, one can give precise equations for $\tau_{\DD/\D}V$ as follows. Suppose $V\subseteq \U ^n$ is affine and let $x$ and $u$ be $n$-tuples of indeterminates; for each $D\in \DD$ and $f\in K\{x\}_\D$ define $d_{D/\D} f\in K\{x,u\}_\D$ as
$$d_{D/\D} f(x,u):= \sum_{\theta\in \Theta_\D, i\leq n}\frac{\partial f}{\partial (\theta x_i)}(x)\; \theta u_i + f^{D}(x)$$
Then, $\tau_{\DD/\D}V$ is a $\D$-closed subset of $\U^{n(r+1)}$ defined by 
$$f(x)=0, \; d_{D_1/\D}f(x,u_1)=0, \; \dots,\;  d_{D_r/\D}f(x,u_r)=0$$
where each $u_j$ is an $n$-tuple of indeterminates and $f$ varies in $\mathcal I(V/K)_\D$ the $\D$-vanishing ideal of $V$ over $K$. 

By a parameterized D-variety w.r.t $\DD/\D$ (aka parameterized $\DD/\D$-variety), one means a pair $(V,s)$ where $V$ is a $\D$-variety and $s=(\operatorname{Id},s_1,\dots,s_r)$ is a $\D$-section of the natural projection $\tau_{\DD/\D}V\to V$ satisfying the following integrability condition: for all $1\leq i<j\leq r$
$$d_{D_i/\D}s_j(a,s_i(a))=d_{D_j/\D}s_i(a,s_j(a))$$
for all $a\in V$. 
A parameterized D-subvariety of $(V,s)$ is a $\D$-subvariety $W\subseteq V$ such that $s(W)\subseteq \tau_{\DD/\D}W$. 
A $\D$-morphism $f:V\to W$ of parameterized $\DD/\D$-varieties $(V,s)$ and $(W,t)$ is said to be a $\DD/\D$-morphism if $\tau_{\DD/\D}(f)\circ s=t\circ f$.

Given a parameterized D-variety $(V,s)$ the $\#$-differential equation on $V$ is defined as
\begin{equation}\label{sharpeq}
\nabla^V_\DD(x)=s(x)
\end{equation}
and the set of $\#$-points of $(V,s)$ is the solution set of \eqref{sharpeq} and is denoted by $(V,s)^\#$. Note that $(V,s)^\#$ is a $\Pi$-variety, and for each $a\in V$ we have that $\D$ bounds the $\Pi$-type$(a/K)$.

For a $K$-irreducible $\Pi$-variety $X$, the $\Pi$-type and typical $\Pi$-dimension of $X$ over $K$ are defined as $\Pi$-type$(a/K)$ and $\Pi$-dim$(a/K)$, respectively, for any $\Pi$-generic point $a\in X$
 over $K$. Similarly, we say $\D$ bounds (respectively, witnesses) the $\Pi$-type of $X$ if it bounds (respectively, witnesses) $\Pi$-type$(a/K)$. From \cite{LS1} we have the following facts.

\begin{fact}\label{Dvarprop}\cite[Proposition 3.10 and 3.12]{LS1} Let $(V,s)$ be a parameterized $\DD/\D$-variety over $K$.
\begin{enumerate}
    \item $(V,s)^\#$ is $\D$-dense in $V$. Moreover, if $V$ is $K$-irreducible, there exists a $\D$-generic of $V$ (over $K$) inside $(V,s)^\#$ and any such point is in fact a $\Pi$-generic of $(V,s)^\#$ over $K$.
    \item Taking $\#$-points defines a 1:1 correspondence between parameterized D-subvarieties of $(V,s)$ over $K$ and $\Pi$-subvarieties of $(V,s)^\#$ over $K$ (where the inverse is given by taking $\D$-closures over $K$). 
    \item Suppose $X$ is a $K$-irreducible $\Pi$-variety and that $\D$ bounds the $\Pi$-type of $X$, then there exists a parameterized $\DD/\D$-variety $(W,t)$ over $K$ such that $X$ is $\Pi$-rationally equivalent (over $K$) to $(W,t)^\#$.
\end{enumerate}
\end{fact}

We now move on to parameterized D-groups and D-torsors with respect to the partition $\DD/\D$. As we pointed out above, $\tau_{\DD/\D}$ commutes with products. Thus, if $G$ is a $\D$-group we get an induced $\D$-group structure on $\tau_{\DD/\D}G$ via 
$$\tau_{\DD/\D}\mu:\tau G\times \tau G\to \tau G$$ 
where $\mu:G\times G\to G$ denotes the multiplication map. Moreover, if $V$ is a (right) $\D$-torsor for $G$, meaning that there is a $\D$-morphism $*:V\times G\to V$ that yields a regular action of $G$ on $V$, then $\tau_{\DD/\D}(*):\tau_{\DD/\D} V\times \tau_{\DD/\D} G\to \tau_{\DD/\D} V$ is a $\D$-morphism that yields a (right) regular action of $\tau_{\DD/
D} G$ on $\tau_{\DD/\D} V$. We will denote the latter action $\tau_{\DD/
\D}(*)$ simply by $\cdot$.

As we noted above, for any $\D$-morphism $f:V\to W$ we have that $\nabla^W_\DD\circ f=\tau_{\DD/\D}f\circ \nabla^V_{\DD}$. Hence, for $\mu:G\times G\to G$ the multiplication map of a $\D$-group $G$, we have 
$$\nabla^G_\DD\circ \mu=\tau_{\DD/\D}\mu\circ \nabla^{G\times G}_{\DD}.$$
In other words, the map $\nabla^G_\DD:G\to \tau_{\DD/\D}G$ is a group homomorphism. More generally, if $V$ is a (right) $\D$-torsor for $G$, with action $*:V\times G\to V$, then
$$\nabla^V_\DD\circ *=\tau_{\DD/\D}(*)\circ \nabla^{V\times G}_{\DD}$$
and so the map $\nabla^V_\DD:V\to \tau_{\DD/\D}V$ preserves the action; namely, 
$$\nabla^V_\DD(v* g)=\nabla^V_\DD(v)\cdot \nabla^G_\DD(g).$$

A parameterized $D$-variety $(G,t)$ with respect to $\DD/\D$ where $G$ is a $\D$-group is called a parameterized $D$-group with respect to $\DD/
\D$ (aka parameterized $\DD/\D$-group) if $t:G\to \tau_{\DD/
\D} G$ is a group homomorphism. Note that being a parameterized $D$-group is equivalent to $(G,t)^\#$ being a subgroup of $G$. 

\begin{comment}
\begin{lemma}
Let $(G,t)$ be a parameterized $\DD/\D$-group. Taking $\#$-points defines a 1:1 correspondence between parameterized D-subgroups of $(G,t)$ over $K$ and $\Pi$-algebraic subgroups of $(G,t)^\#$ over $K$.
\end{lemma}
\begin{proof}
By Fact~\ref{Dvarprop}(ii), it suffices to show that the $\D$-closure over $K$ of a $\Pi$-algebraic subgroup $H$ of $(G,t)^\#$ is a subgroup of $G$. Consider 
$$f_1(x\cdot y^{-1})=e, \dots, f_s(x\cdot y^{-1})=e$$
where $e$ is the identity of $G$ and the $f_i$'s are the $\D$-equations defining the $\D$-closure $\bar H$ of $H$. These equations hold for all $(x,y)\in H\times H$. Now, as $H\times H$ is $\D$-dense in $\bar H\times \bar H$, this also holds in $\bar H\times \bar H$, and hence $\bar H$ is a subgroup of $G$.
\end{proof}
\end{comment}

More generally, given a parameterized $D$-group $(G,t)$ and a parameterized $D$-variety $(V,s)$ such that $V$ is a (right) $\D$-torsor for $G$, we say that $(V,s)$ is a parameterized (right) $D$-torsor for $(G,t)$ (w.r.t. $\DD/\D$) if $s$ preserves the action $*$; more precisely, 
$$s(v*g)=s(v)\cdot t(g)$$
where recall that $*$ denotes the action of $G$ on $V$ and $\cdot$ denotes the action of $\tau G$ on $\tau V$. Note that being a parameterized $D$-torsor is equivalent to $(V,s)^\#$ being a $\Pi$-torsor for $(G,t)^\#$ with respect to the induced action $*$ of $G$ on $V$. For ease of notation, we will write $(V,s)\circlearrowright (G,t)$ when $(V,s)$ is a paramtererized $\DD/\D$-torsor for $(G,t)$.

%\begin{remark}
%In a paper of Wibmer et al, they talk about $D$-torsors in the ordinary case when the group $G$ is over the constants equipped with the zero section.  We should briefly talk about this.
%\end{remark}

%From \cite[\S 4]{LS1}, we obtain the following equivalence of categories.

Recall from Fact~\ref{Dvarprop} that, via the $\#$-points functor, there is a close connection between parameterized $\DD/\D$-varieties and $\Pi$-varieties with $\D$ bounding their $\Pi$-type. In the case of parameterized D-groups and D-torsors, we have a much closer relationship.

\begin{theorem}\label{thm: equiv of cats} %Let $(K,\Pi)$ be a partial differential field. Let $\Pi = \DD \cup \D$ be a disjoint partition of the set of derivations with $\DD$ nonempty.
\
\begin{enumerate}
    \item There is an equivalence of categories between connected parameterized $\DD/\D$-groups over $K$ and connected $\Pi$-algebraic groups over $K$ with $\Pi$-type bounded by $\D$. The functor is given by taking $\#$-points in the former category. 
    \item More generally, let $\mathcal E$ be the category of parameterized $\DD/\D$-torsors for connected $\DD/\D$-groups over $K$ where morphisms between $(V,s)\circlearrowright (G,t)$ and $(W,r)\circlearrowright(G,t)$ are $\DD/\D$-morphisms $(V,s)\to (W,r)$ over $K$ that preserve the action. On the other hand, let $\mathcal F$ be the category of $\Pi$-torsors for connected $\Pi$-algebraic groups over $K$ with $\Pi$-type bounded by $\D$ where the morphisms between $X\circlearrowright H$ and $Y\circlearrowright  H$ are $\Pi$-morphisms $X\to Y$ over $K$ preserving the action. Then, the $\#$-points functor yields an equivalence between $\mathcal E$ and $\mathcal F$.
\end{enumerate}
\end{theorem}
\begin{proof}
Part (1) follows from \cite[Theorem 4.6]{LS1} and, in a natural manner, (2) follows from (1). 
\end{proof}

%We quote one final fact from \cite{LS1} which relates the model-theoretic notion of internality to $\Pi$-type.

%\begin{fact}\cite[Lemma 1.9]{LS1}\label{fact: pitype internality}
   % Let $X$ be a K-definable set and let $a$ be a tuple from $\mathcal{U}$. Suppose $\tp(a/K)$ is $X$-internal, then
      %  \begin{enumerate}
        %    \item $\Pi$-$type(a/K) \leq \Pi$-$type(X)$.
          %  \item If $\D$ bounds the $\Pi$-$type(X)$ then $\D$ also bounds the $\Pi$-$type(a/K)$.
       % \end{enumerate}
%\end{fact}

We now recall the definition of an $X$-strongly normal extension initially defined in \cite{LS1} as Definition 2.1. However following \cite{LSPillay}, we weaken the condition of the ``constants being suitably closed" (i.e., $X(K) = X(K^{\diff})$) to just being a ``no-new-constants" extension (i.e., $X(L) = X(K)$). 

\begin{definition}\cite[Definition 3.3]{LSPillay}
Let $X$ be a $K$-definable set. A finitely generated differential field $L/K$ is an $X$-strongly normal extension of $K$ if 
\begin{enumerate}
    \item For any $\sigma \in \Aut(\mathcal{U}/K)$ we have $\sigma(L) \subset L\l X\r$.
    \item $X(L) = X(K)$.
\end{enumerate}
   A differential extension of $K$ is called generalized strongly normal if it is $X$-strongly normal for some $K$-definable $X$.
\end{definition}

\begin{remark} \
\begin{enumerate}
    \item The fundamental theorem of Galois theory was first proved in \cite[Theorem 2.7]{LS1}. Under the present no-new-constants hypothesis, the fundamental theorem is in \cite[Theorem 2.3]{LSPillay}, which we re-stated in part (iii) of Fact~\ref{fact: torsor theorem}.
    \item Part (3) of Theorem 2.5 in \cite{LS1} states (in our notation) that if $(G,t)^\#$ is the Galois group of an $X$-strongly normal extension $L/K$, then $L/K$ is also $(G,t)^\#$-strongly normal.  
\end{enumerate}
\end{remark}

We now introduce the notion of a differential Galois extension for a parameterized D-torsor. 

\begin{definition}\label{def: extension for torsor}
Let $(V,s)\circlearrowright (G,t)$ be a parameterized $\DD/\D$-torsor. A differential extension $L/K$ is said to be a differential Galois extension for $(V,s)$ if 
\begin{enumerate}
    \item [(i)] $L=K\langle a\rangle_\Pi$ for $a$ a solution of the $\#$-differential equation on $V$, recall this is $\nabla^V_\DD(x)=s(x)$ as defined in \eqref{sharpeq}, and 
    \item [(ii)] $K\langle (G,t)^\#\rangle_\Pi\cap L=K$ or equivalently $(G,t)^\#(L) = (G,t)^\#(K)$.
\end{enumerate}
\end{definition}

We now observe that these extensions are generalized strongly normal extensions with Galois group embedding in $(G,t)^\#$, that is, $(G,t)^\#$-strongly normal extensions. 

%More precisely, a solution to a $\#$-differential equation on a parameterized D-torsor $(V,s)$ generates a $(G,t)^\#$-strongly normal extension if and only if it generates a differential Galois extension for $(V,s)$ if and only if part $(ii)$ of Definition~\ref{def: extension for torsor} holds:

\begin{lemma}\label{lem: part a}
    If $L/K$ is a differential Galois extension for a parameterized $\DD/\D$-torsor $(V,s)\circlearrowright(G,t)$, then $L/K$ is a $(G,t)^\#$-strongly normal extension. 
    
    Conversely, if a solution of the $\#$-differential equation on a parameterized $\DD/\D$-torsor $(V,s)\circlearrowright(G,t)$ generates a $(G,t)^\#$-strongly normal extension, then the extension is a differential Galois extension for $(V,s)$.
\end{lemma}
\begin{proof}
Let $L/K$ be a differential Galois extension for a parameterized torsor. Then by Condition (i) we have $L=K\langle a \rangle_\Pi$ such that $\nabla^V_\DD(a)=s(a)$. We argue that $p:=tp(a/K)$ is strongly $(G,s)^\#$-internal. Indeed, if $b\models p$, then $\nabla^V_\DD(b)=s(b)$. Now let $g\in G$ such that $b=a*g$ (this is by transitivity of the action $*$). It suffices to show that $g\in (G,t)^\#$. We have $s(b)=s(a)\cdot t(g)$, and on the other hand
$$s(b)=\nabla^V_\DD(b)=\nabla^V_\DD(a*g)=\nabla^V_\DD(a)\cdot \nabla^G_\DD(g)=s(a)\cdot \nabla^G_\DD(g)$$
 by regularity of the action $\cdot$, this yields $\nabla^G_\DD(g)=t(g)$ as claimed. Condition (ii) is just a translation of $p$ being weakly orthogonal to $(G,t)^\#$.

For the converse, let $a$ be a solution of the $\#$-equation. If $L = K\langle a \rangle_\Pi$ is $(G,t)^\#$-strongly normal, then $\tp(a/K)$ is weakly orthogonal to $(G,t)^\#$. As we said, this is synonomous with the desired condition (ii) which was all we were missing. 
\end{proof}

%{\bf We should observe how the Galois correspondence is now with paramaterized D-subgroups of $G$}

Let $L=K\langle a\rangle_\Pi$ be a differential Galois extension for a parameterized $\DD/\D$-variety $(V,s)\circlearrowright(G,t)$. Let $\mathcal G=Aut_\Pi(L\langle(G,t)^\#\rangle/K\langle(G,t)^\#\rangle)$ be its (differential) Galois group. For any $\sigma\in \mathcal G$, we have that $\sigma(a)\in (V,s)^\#$. Set $\mu(\sigma)$ be the unique element of $(G,t)^\#$ with $\sigma(a)=a*\mu(\sigma)$. It is straightforward to check that $\mu:\mathcal G\to (G,t)^\#$ is an injective group homomorphism and hence $\mathcal G$ is a isomorphic to the $\#$-points of a $D$-subgroup of $(G,t)$. In fact, the arguments of \cite[Proposition 5.8 and Corollary 5.9]{LS1} yield:

\begin{proposition}
    Let $L=K\langle a\rangle_\Pi$ be a differential Galois extension for a parameterized $\DD/\D$-variety $(V,s)\circlearrowright(G,t)$. The Galois group of $q=tp(a/K)$ is of the form $(H,t)^\#$ for some parameterized D-subgroup $H$ of $(G,t)$ over $K$ and the regular action is given by $b*h$ for $b\models q$ and $h\in (H,t)^\#$. Furthermore, there is a natural 1-1 correspondence between intermediate $\Pi$-fields of $L/K$ and parameterized D-subgroups of $(H,t)$ defined over $K$.
\end{proposition}

In Pillay's paper \cite{Pillay2}, or the first author's \cite{LS1}, differential Galois extensions for parameterized D-torsors were not explored (at this level of generality); only a special case of them were studied: those coming from log-differential equations on parameterized D-groups. We now recall these special case of differential Galois extensions. Given a $\D$-group $G$ the fibre of $\pi:\tau_{\DD/\D}G\to G$ at the identity $e\in G$ is denoted by $\tau_{\DD/\D}G_e$. Given a parameterized $\DD/\D$-group $(G,t)$, we say that $\alpha\in\tau_{\DD/\D}G_e$ is an integrable point of $(G,t)$ if $(G,\alpha \, t)$ is a parameterized $\DD/\D$-variety; here $\alpha \, t:G\to \tau_{\DD/
\D}G$ denotes the section $\alpha\cdot t(x)$. From \cite[Lemma 5.3]{LS1}, we see that $\alpha$ is an integrable point if and only if 
\begin{equation}\label{log-eq}
    \nabla^G_\DD(x)=\alpha \cdot t(x)
\end{equation}
has a solution (in $G$). We refer to \eqref{log-eq} as a log differential equation on $(G,t)$, or $\alpha$-log differential equation if we wish to specify $\alpha$.

\begin{definition}
Let $(G,t)$ be a parameterized $\DD/\D$-group over $K$.
%and let $\alpha\in \tau_{\DD/\D}G_e(K)$ be an integral $K$-point. 
A differential extension $L/K$ is said to be a differential Galois extension for a log differential equation on $(G,t)$ if the following two conditions hold:
\begin{enumerate}
    \item [(i)] There exists a $K$-point $\alpha\in \tau_{\DD/\D}G_e(K)$ such that $L=K\langle a\rangle_\Pi$ for $a$ a solution of the $\alpha$-log differential equation $\nabla^G_\DD(x)=\alpha \cdot t(x)$
    \item [(ii)] $K\langle (G,t)^\#\rangle_\Pi\cap L=K$, or equivalently, $(G,t)^\#(L) = (G,t)^\#(K)$.
\end{enumerate}
\end{definition}

\begin{lemma}
    If $L=K\langle a\rangle_\Pi$ is a differential Galois extension for an $\alpha$-log differential equation on a parameterized D-group $(G,t)$, then $L/K$ is a differential Galois extension for the $\#$-differential equation on a parameterized D-torsor.
\end{lemma}
\begin{proof}
    Let $(V,s)$ with $V=G$ and $s=\alpha\cdot t$. Since the equation $\nabla^G_\DD(x)=\alpha \cdot t(x)$ has a solution (namely $a$), the section $s:V\to \tau_{\DD/\D}V$ is integrable (see \cite[Lemma 5.3]{LS1}), and hence $(V,s)$ is indeed a parameterized D-variety. It is clearly a $D$-torsor for $(G,t)$ and $a$ is a $\#$-point of $(V,s)$.
\end{proof}

 In Remark~\ref{rem:example} we saw that a differential Galois extension $L/K$ for a parameterized D-torsor $(V,s)\circlearrowright (G,t)$ with Galois group $(G,t)^\#$ need not be a Galois extension for a log-differential equation on $(G,t)$.  

%that The main example from the introduction shows that the converse does not hold in general; namely, not all differential Galois extensions of a $\#$-differential equation on a parameterized torsor are the differential extension for a $\alpha$-log equation on a parameterized group.

Our main result,  in Section~\ref{sec:mainresult}, shows that any generalized strongly normal extension is the differential Galois extension for a parameterized $D$-torsor. Furthermore, we will characterize when the extension is the differential Galois extension for a log-differential equation on its Galois group.

\

\section{A parameterized version of Kolchin's cohomology theorem}\label{sec:parKoltheorem}

We carry forward the setup from the previous section. In particular, we have a partition of the derivation set as 
$$\Pi=\Delta\cup \DD.$$
We work inside a saturated model $(\U, \Pi)\models \DCF_{0,m}$ and $K$ denotes a (small) differential subfield. By $\HH^1_{\Pi}$ we mean definable Galois cohomology in $(\U,\
\Pi)\models\DCF_{0,|\Pi|}$; while by $\HH^1_{\D}$ we mean definable cohomology in the reduct $(\U,\D)\models\DCF_{0,|\D|}$.

Kolchin's cohomology theorem from \cite[\S VII.3]{Kolchin2} states that for an algebraic $G$ over $K$ we have $\HH^1_\Pi(K,G)\cong \HH^1_{alg}(K,G)$ where the latter is (algebraic) Galois cohomology which coincides with the definable cohomology in the reduct to $\U\models ACF_0$ (namely when taking $\D$ to be empty). The aim of this section is to prove a parameterized version of this when the group $G$ is a $\D$-algebraic group. We will use the following facts, see \cite[Corollary 4.2]{Pillay4} and \cite[Fact 1.5]{Pillay5}:

\begin{fact}\label{embedding} Let $G$ be a $\Pi$-algebraic group over $K$.
    \begin{enumerate}
        \item There exists an algebraic group $H$ over $K$ and a $K$-definable group embedding $G\hookrightarrow H$. We identify $G$ as a subgroup of the algebraic group $H$.
        \item In addition, if $V\circlearrowright G$ is a $\Pi$-torsor over $K$, then there exists an algebraic torsor $W\circlearrowright H$ and a $K$-definable embedding $f:V\hookrightarrow W$ such that
        $$f(v* g)=f(v)\cdot g \quad \text{ for all } v\in V, g\in G$$
        where $*$ is the action in $V\circlearrowright G$ while $\cdot$ is the action in $W\circlearrowright H$. In other words, $f(V)\circlearrowright G$ is a subtorsor of $W\circlearrowright H$; i.e., the action of $G$ on $f(V)$ is the restriction of the action of $H$ on $W$.
    \end{enumerate}
\end{fact}

We can now prove the parameterized version of Kolchin's cohomology theorem.

\begin{theorem}\label{thm: param kolchin}
    Suppose $G$ is a $\D$-algebraic group over $K$. Then, there is a canonical isomorphism
    $$\HH^1_\Delta(K,G)\cong \HH^1_\Pi(K,G).$$
\end{theorem}

\begin{proof}
For $\alpha\in \HH^1_\Delta(K,G)$, we may write $\alpha=[X]_\Delta$ where $X$ is a $\Delta$-torsor over $K$ and $[X]_\Delta$ denotes the class of $\Delta$-torsors over $K$ that are $\D$-isomorphic to $X$ over $K$. Of course, such an $X$  is also a $\Pi$-torsor and we may similarly write $[X]_\Pi$ (for the class of $\Pi$-torsors over $K$ that are $\Pi$-isomorphic to $X$ over $K$). We define 
$$\phi: \HH^1_\Delta(K,G)\to \HH^1_\Pi(K,G)$$
in the most natural (or canonical) way, that is
$$\alpha = [X]_\Delta\mapsto [X]_\Pi.$$
Note that $\phi$ is well defined. Indeed, if $Y$ is a $\D$-torsor over $K$ which is $\D$-isomorphic to $X$ and $V$ is a $\Pi$-torsor $\Pi$-isomorphic to $X$, then $Y$ and $V$ are $\Pi$-isomorphic (as $\Pi$-torsors). It follows that if $[X]_\D=[Y]_\D$ then $[X]_\Pi=[Y]_\Pi$.

We now prove $\phi$ is injective. Assume $\phi(\alpha)=\phi(\beta)$ with $\alpha=[X]_\D$ and $\beta=[Y]_\D$. In other words, we are assuming $[X]_\Pi=[Y]_\Pi$; namely, there is an $\Pi$-isomorphism $f:X\to Y$. To show injectivity it suffices to show that there is a $\D$-isomorphism over $K$.

Using Fact~\ref{embedding}, there are algebraic torsors $V\circlearrowright H$ and $W\circlearrowright H$ such that we may view $X\circlearrowright G$ as a substorsor of $V\circlearrowright H$ and $Y\circlearrowright G$ as a substorsor of $W\circlearrowright H$.
Let $h:V\times V \to H$ be the function taking $(x,y)$ to the unique element of $H$ mapping $x\mapsto y$ (note that this is ACF-definable). Fix $a\in X$ and consider the map $\rho:X\to Y$ given by 
$$\rho(x)=f(a)\cdot h(a,x)$$
where $\cdot$ is the action of $H$ on $W$ (so ACF-definable). Note that, since $X$ and $Y$ are $\D$-definable, the map $\rho$ is $\D$-definable. Also, a straightforward computation shows that it is a torsor isomorphism. All that remains to show is that $\rho$ is definable over $K$. To see this, it suffices to show that for all $\s\in Aut_{\Pi}(\U/K)$ we have $\sigma(\rho(x))=\rho(\sigma(x))$ for all $x\in X$. We have
\begin{align*}
 \sigma(\rho(x)) & = \sigma(f(a)h(a,x)) \\
 & = f(\sigma(a))h(\sigma(a), \sigma(x)) \\
 & = f(\sigma(a)h(\sigma(a), \sigma(x))) \\
 & = f(ah(a,\sigma(x))) \\
 & = f(a) h(a,\sigma(x)) \\
 & = \rho(\sigma(x))
\end{align*}
as desired. 

We now prove surjectivity. Let $\mu\in \HH_\Pi^1(K,G)$ then $\mu=[U]_\Pi$ for $U$ a $\Pi$-torsor. We need to produce a $\D$-torsor $\Pi$-isomorphic to $U$ over $K$. By Fact~\ref{embedding}, there is a $K$-definable embedding from $U$ to an algebraic torsor $V$ for $H$ (an algebraic group where $G$ is embedded). Let $X$ be the image of $U$ inside $V$. We will be done once we show that $X$ is $\D$-definable. Consider the equivalence relation on $V$
$$a\sim b\quad \iff \quad \text{there is $g\in G$ such that $a\cdot g=b$}$$
since $G$ is $\D$-definable and the action $a\cdot g$ above is ACF-definable, we have that the above is a $\D$-definable equivalence relation. Hence, we may identify $V/G$ with a $\D$-definable set over $K$ and the quotient map 
$p:V\to V/G$
is $\D$-definable over $K$. Since $X$ is an equivalence class, the point $\gamma=p(X)\in V/G$ is a $K$-point. Thus $X=p^{-1}(\gamma)$ is $\D$-definable over $K$, as desired.
\end{proof}

\begin{remark} 
%\begin{enumerate}
    %\item 
    Let us note that the map above $\phi: \HH^1_\Delta(K,G)\to \HH^1_\Pi(K,G)$ can also be described in terms of definable cocycles. Namely, let 
    $$s:\Aut_{\D}(K^{\D\text{-}\diff}/K)\to G(K^{\D\text{-}\diff})$$ 
    be a $\D$-definable cocycle, say definability data $h(x,y)$ and $a$. Let $\phi(x)$ be a formula isolating $tp_\D(a/K)$. Then we know that for all $b,c,d$ realising $\phi$ in $K^{\D\text{-}\diff}$, we get that $h(b,c)$ is defined and in $G(K^{\D\text{-}\diff})$ and
    \begin{equation}\label{cocyclecondition1}
    h(b,c)h(c,d)=h(b,d)
    \end{equation}
    Define $t:\Aut_{\Pi}(K^{\Pi\text{-}\diff}/K)\to G(K^{\Pi\text{-}\diff})$ by
    $$t(\sigma)=h(a,\sigma(a))$$ 
    recalling also that we may take $K^{\D\text{-}\diff} \leq K^{\Pi\text{-}\diff}$. It is easy to check, using \eqref{cocyclecondition1}, that $t$ is indeed a $\Pi$-definable cocycle: namely, any realization of $\tp_{\Pi}(a/K)$ must satisfy $\tp_{\D}(a/K)$. Now take the class of this cocycle in $\HH_\Pi^1(K,G)$. 

\end{remark}

\begin{notation}\label{not: themap}
    Given a parameterized $\DD/\D$-group $(G,t)$ over $K$, we will denote by
    $$\Phi:\HH^1_\Pi(K,(G,t)^\#)\to \HH^1_\D(K,G)$$
    the map obtained by composing the induced map $\HH^1_{\Pi}(K,(G,t)^\#) \xrightarrow{\iota^1} \HH^1_{\Pi}(K,G)$ from Fact~\ref{fact: LES} with the isomorphism $\HH^1_\Pi(K,G)\to \HH^1_\D(K,G)$ obtained in Theorem~\ref{thm: param kolchin}.
\end{notation}

We now observe that the equivalence of categories between parameterized $\DD/\D$-torsors and $\Pi$-torsors with $\Pi$-type bounded by $\D$ stated in Theorem~\ref{thm: equiv of cats} gives the following cohomological result:

\begin{lemma}\label{cor: coho conseq. proved}
     %Let $H$ be a $\Pi$-differential algebraic group defined over $K$, that is a $K$-definable group in $DCF_{0,m}$. Assume that $\D \subset \Pi$ with $|\D| < |\Pi|$ bounding the $\Pi$-type of $H$. Then $\D$ can be completed as a basis $\DD\cup \D$ of $span_{K^{\Pi}}\Pi$ so that there is a parametric $\D/\DD$-group $(G,t)$ with $H \cong (G,t)^\#$ and we have a exact sequence 
     Let $(G,t)$ be a connected parameterized $\DD/\D$-group. Then, the equivalence of categories from Theorem~\ref{thm: equiv of cats} yields a natural map 
     $$\Lambda:\HH^1_{\Pi}(K,(G,t)^\#) \to \HH^1_{\D}(K,G)$$
    %$$(G/(G,t)^\#)(K) \xrightarrow{\delta^1} \HH^1_{\Pi}(K,(G,t)^\#) \xrightarrow{\iota} \HH^1_{\D}(K,G)$$ 
\end{lemma}

\begin{proof}
    Let $Q$ be a $\Pi$-algebraic torsor for $(G,t)^\#$ over $K$. Since $\D$ bounds the $\Pi$-type of $(G,t)^\#$ and $Q$ is internal to $(G,t)^\#$, by \cite[Lemma 1.9]{LS1}, we have that $\D$ also bounds the $\Pi$-type of $Q$. We can then apply Theorem \ref{thm: equiv of cats} to see that there is a parameterized $\DD/\D$-torsor $(V,s)\circlearrowright (G,t)$ with $(V,s)^\# \cong Q$ as $\Pi$-torsors over $K$. We then define $\Lambda$ by sending the class $[Q]_\Pi$ to the class of $[V]_\D$, that this map is well-defined uses again the equivalence of categories in Theorem~\ref{thm: equiv of cats}. 

\end{proof}

%\begin{remark}\label{rmk: ses application}
%    Let $L/K$ be a differential Galois extension for a $\DD/\D$-parametric torsor $(V,s)$ for a $\DD/\D$-parametric group $(G,t)$. We have an injection 
%    $$1 \to \HH^1_{\Pi}(L/K,(G,t)^\#(L)) \to  \HH^1_{\Pi}(K^{\diff}/K,(G,t)^\#(K^{\diff}))$$ and the class of  $H^1_{\Pi}(K^{\diff}/K,(G,t)^\#(K^{\diff}))$ representing the torsor $(V,s)^\#(K^{\diff})$ is in the image of this injection. 
%\end{remark}

%\begin{proof}
%    We apply the short exact sequence in Kolchin's constrained differential Galois cohomology, appearing as Theorem 1 of Chapter 7 Section 2 of \cite{Kolchin2}, or the short exact sequence in definable Galois cohomology in $DCF_{0,m}$, appearing as Theorem 4.3 of \cite{Me} to the normal extension $L/K$ in $K^{\diff}$.

%    By definition $L$ is generated by a point of $(V,s)^\#$, so the torsor $(V,s)^\#$ contains an $L$-point and this means exactly that the associated cohomology class is in the image of the above injection. 
%\end{proof}

We conclude this section by pointing out that the two maps defined above from $\HH^1_\Pi(K,(G,t)^\#)\to \HH^1_\D(K,G)$ actually coincide.

\begin{lemma}
Let $(G,t)$  be a connected parameterized $\DD/\D$-group. Then, the two maps above $\Phi$ and $\Lambda$ from 
$$\HH^1_\Pi(K,(G,t)^\#)\to \HH^1_\D(K,G)$$
in Notation~\ref{not: themap} and Lemma~\ref{cor: coho conseq. proved}, respectively,
coincide. 
%More precisely, let $\Psi:\HH^1_\Pi(K,G^\#)\to \HH^1_\Pi(K,G)$ be the usual map in cohomology (given by the inclusion of $G^\#$ in $G$) and let $\chi$ be the inverse of the map from Theorem~\ref{thm: param kolchin}. Also let $\iota$ be the map from Lemma~\ref{cor: coho conseq.}. We then have that $\iota=\chi\circ\Psi$.
\end{lemma}
\begin{proof}
Let $Q$ be a $\Pi$-torsor for $(G,t)^\#$. We have seen in the proof of Lemma~\ref{cor: coho conseq. proved} that we may assume that $Q=(V,s)^\#$ for some parameterized $\DD/\D$-torsor $(V,s)\circlearrowright (G,t)$. Then, by definition, $\Lambda([Q]_\Pi)=[V]_\D$. We must show that $\Phi([Q]_\Pi)=[V]_\D$. Since $Q\circlearrowright (G,t)^\#$ is a $\Pi$-subtorsor of $V\circlearrowright G$, by part (4) of Remark \ref{rmk: subtorsor}, the induced map in cohomology $\iota^1$ maps $[Q]_\Pi$ to $[V]_\Pi$. Noting that the isomorphism $\HH^1_\Pi(K,G)\to \HH_\D^1(K,G)$ from Theorem~\ref{thm: param kolchin} maps $[V]_\Pi$ to $[V]_\D$, we see that $\Phi([Q]_\Pi)$ is indeed $[V]_\D$. 

%$\Psi(X)$ is $\Pi$-isomorphic to $V$. 

%We repeat the argument just to be self contained: Let us recall how the map $\Psi$ is defined at the level of torsors (at the level of cocycles it is obvious, as any cocycle will have its image in $G^\#\subseteq G$). For any $a\in X$ in a $\Pi$-closure of $K$, one considers: let $Z$ be the realisations of $tp^\Pi(a/K)$ and as usual consider 
%$$Z\times G/\sim$$
%where $\sim$ is the usual equivalence relation; namely. $(a,g)\sim (b,k)$ iff $g=h(a,b)k$ where $h(a,b)$ is the unique element of $G^\#$ mapping $a\mapsto b$. In fact we can take $h$ being induced by the torsor structure on $V$, that is, for $x,y\in V$, we set $h(x,y)$ to the unique element of $G$ mapping $x\mapsto y$. We see now that $\Psi(X)=[Z\times G/\sim]_\Pi$.

%Now consider the map $\rho:Z\times G\to V$ given simply by $(a,g)\mapsto a\cdot g$. Clearly $\rho$ is onto. One can now easily check that $\rho$ preserves the equivalence relation and hence induces an $\Pi$-isomorphism over $K$ between $Z\times G/\sim$ and $V$; in other words, $\Psi(X)=[V]_\Pi$, as desired.
\end{proof}

%\begin{remark}
%    Note that the argument in the above lemma can be used to show the following: if $(V,G)$ is a definable subtorsor of $(W,H)$ then $\Psi(V)=[W]_{\defin}$ (where $\Psi:\HH^1_{\defin}(K,G)\to \HH^1_{\defin}(K,H)$ is the usual cohomology map with $G\leq H$). %Perhaps we can have this in the general model-theoretic section??
%\end{remark}

\section{The main result}\label{sec:mainresult}

We recall the setting. We are working in $\DCF_{0,m}$ where $\Pi$ is the set of $m$ commuting derivations. 
%We fix a sufficiently saturated model $(\mathcal{U},\Pi)$ and a small differential subfield $K$. 
Recall that for $\D \subseteq \Pi$, by $\HH^1_{\D}$ we mean definable Galois cohomology in the reduct $\DCF_{0,|\D|}$.

We now prove the main theorem of the paper.

\begin{theorem}%[{\bf Theorem A: Main result}]
\label{thm:mainresult}
Let $(\U,\Pi)$ be a monster model of DCF$_{0,m}$ and $K$ a (small) differential subfield.
%Let $K$ be a partial differential field. 
\begin{enumerate}
    \item Let $\DD\cup\D$ be a partition of $\Pi$ with $\DD$ nonempty and $(V,s)\circlearrowright (G,t)$ a parameterized $\DD/\D$-torsor. A solution $a$ to the $\#$-differential equation of $(V,s)$ generates a $(G,t)^\#$-strongly normal extension $L=K\langle a \rangle_\Pi$ if and only if $(G,t)^\#(L)=(G,t)^\#(K)$.
    %Let $L = K\langle a \rangle_\Pi$ be a differential extension where $a$ to the $\#$-equation on a parametric $\DD/\D$-torsor $(V,s)$ for $(G,t)$, a parametric $\DD/\D$-group defined over $K$. The following are equivalent:
        %\begin{enumerate}
         %   \item $L/K$ is a differential Galois extension of the torsor $(V,s)\circlearrowright(G,t)$
          %  \item $(G,t)^\#(L) = (G,t)^\#(K)$
           % \item $\tp(a/K)$ is weakly orthogonal to $(G,t)^\#$
           % \item $L/K$ is a $(G,t)^\#$-generalized strongly normal extension.
        %\end{enumerate}
    \item Let $L$ be a generalized strongly normal extension of $K$ %Let $L = K \l b \r_\Pi$ be a generalized $X$-strongly normal extension 
    with $L/K$ regular (so the differential Galois group is connected). Then, after a $GL_m(\mathbb Q)$-transformation of $\Pi$, there exists a partition $\DD\cup\D$ with $\D$ witnessing the $\Pi$-type of $L/K$ and a parameterized $\DD/\D$-torsor $(V,s)\circlearrowright (G,t)$ over $K$ such that $(G,t)^\#$ is the differential Galois group of $L/K$ and $L$ is a differential Galois extension for $(V,s)$. In fact, $L$ is isomorphic the $\Pi$-function field of $(V,s)^\#$. 
    %Let $\D$ be a set of linearly independent elements of $\text{span}_{K^\Pi}(\Pi)$ such that $\D$ bounds the $\Pi$-type of $X$ and assume $|\D| < |\Pi|$. Then we can extend $\D$ to a basis $\DD \cup \D$ of $\text{span}_{K^\Pi}(\Pi)$ and there exists a parameterized $\DD/\D$-torsor $(V,s)$ for $(G,t)$ and a logarithmic differential equation 
    %$$(*) \quad \nabla^V_\DD(x) = s(x)$$ 
    %such that $L=K\l a \r$ for $a$ a solution of (*). In other words, $L$ is a differential Galois extension for the logarithmic differential equation (*) on $(V,s)$.
    
    Moreover, $L/K$ is the differential Galois extension for a log-differential equation on $(G,t)$ if and only if the image of the class associated to the $\Pi$-torsor $(V,s)^\#\circlearrowright (G,t)^\#$
    %the torsor produced from $L/K$ via the torsor theorem, Fact \ref{fact: torsor theorem} (2), is trivial 
    under the map (from Notation~\ref{not: themap})
    %$$\iota^1:\HH^1_{\Pi}(K,(G,t)^\sharp) \to \HH^1_{\Pi}(K,G)$$ arising from Fact \ref{fact: LES}, the long exact sequence in definable Galois cohomology, or equivalently by Theorem \ref{thm: param kolchin}, the class is trivial under the map 
    $$\Phi:\HH^1_{\Pi}(K,(G,t)^\#) \to \HH^1_{\D}(K,G)$$
    is trivial. In particular, this occurs when $(K,\D)$ is $\D$-closed as then $\HH^1_\D(K,G)$ is trivial.
    %given by Corollary \ref{cor: coho conseq.} 
    %if and only if .
    %generated by a solution of a logarithmic differential equation on $(G,t)$, $$(**) \quad \nabla^G_\DD(x) = \alpha \cdot t(x)$$ with $\alpha \in \tau_{\DD/\D} G_e(K)$, an integral $K$-point of $(G,t)$.
\end{enumerate}
\end{theorem}

\begin{proof} \quad 
\begin{enumerate}
    \item Part (1) is exactly Lemma \ref{lem: part a}.
    
    \item Let $L = K\l b \r_\Pi$ and let $Q$ be the realizations of $q:=tp_\Pi(b/K)$. By the Torsor Theorem (see Fact~\ref{fact: torsor theorem}), the Galois group of $L/K$ is isomorphic to a connected $\Pi$-algebraic group $H$ over $K$ and $Q$ is a $\Pi$-torsor for $H$ over $K$. Moreover, since $q$ is isolated (see Fact~\ref{fact: torsor theorem}(1)), the $\Pi$-type of $Q$ is smaller than $|\Pi|$. Hence, by Fact~\ref{Koltype}, after a $GL_m(\mathbb Q)$-transformation of $\Pi$, there exists $\D$ a proper subset of $\Pi$ witnessing the $\Pi$-type of $Q$ (which is the same as the $\Pi$-type of $L/K$). Let $\DD=\Pi\setminus \D$. Now, by Theorem~\ref{thm: equiv of cats}, $Q$ is of the form $(V,s)^\#$ for some parameterized $\DD/\D$-torsor $(V,s)\circlearrowright (G,t)$ with $(G,t)^\#=H$. Since $b\in (V,s)^\#$, $L$ is clearly a differential Galois extension for $(V,s)$; and, in fact, $L$ is the $\Pi$-function field of $(V,s)^\#$ since $b$ is $\Pi$-generic in $(V,s)^\#$ over $K$.

    We now prove the 'moreover' clause. By Theorem~\ref{thm: cohomological fact}, the image of the class of $(V,s)^\#$ under the map $\Phi$ is trivial if and only if $(V,s)^\#$ is $\Pi$-isomorphic over $K$ to a $K$-definable left coset of $(G,t)^\#$ in $G$. Under this isomorphism, let $a$ be the point in $G$ corresponding to $b$. Then, $(V,s)^\#\cong a\cdot (G,t)^\#$ and $L=K\langle a\rangle_\Pi$.

    Since $a\cdot (G,t)^\#$ is $K$-definable, the image of this coset under the $K$-definable map
    $$\nabla_\DD^G(x)\cdot t(x)^{-1}:G\to \tau_\D G_e$$ 
    must be a $K$-point, call it $\alpha$. Since $a$ is a solution of $\nabla_\DD^G(x) = \alpha \cdot t(x)$, we see that $L$ is a Galois extension of this $\alpha$-log differential equation on $(G,t)$. 

\end{enumerate}
\end{proof}


\begin{thebibliography}{10}

%\nocite{*}


\bibitem{CaSi}
P. Cassidy and M. Singer.
\newblock Galois theory of parametrized differential equations and linear differential algebraic groups.
\newblock Differential Equations and Quantum Groups (IRMA Lectures in Mathematics and Theoretical Physics Vol. 9). EMS Publishing house, pp.113- 157, 2006.


\bibitem{ChPi}
Z. Chatzidakis and A. Pillay.
\newblock Generalized Picard--Vessiot extensions and differential Galois cohomology.
\newblock Annales de la Faculté des sciences de Toulouse : Mathématiques (6), 
Vol.~28, no.~5, pp.~813--830, 2019.


\bibitem{GiGoOv} 
H. Gillet, S. Gorchinskiy, and A. Ovchinnikov, 
\newblock Parametrized Picard-Vessiot extensions and Atiyah extensions.
\newblock Advances in Math., 238 (2013), 322-411.

\bibitem{Hungerford}
T. W. Hungerford.
\newblock Algebra. 
\newblock Graduate texts in Mathematics 73, Springer, 1974.

\bibitem{JuanLedet}
L. Juan and A. Ledet.
\newblock Equivariant vector fields on nontrivial $SO_n$-torsors and differential Galois theory.
\newblock Journal of Algebra, 312(2):734--745, 2007.

%\bibitem{MoshePillay} 
%M. Kamensky and A. Pillay.
%\newblock Interpretations and differential Galois extensions.
%\newblock International Mathematics Research Notices, 2016(24):7390--7413,  2016.

\bibitem{Kolchin1}
E. R. Kolchin.
\newblock Differential algebra and algebraic groups.
\newblock Academic Press. New York, New York 1973.

\bibitem{Kolchin2}
E. R. Kolchin.
\newblock Differential algebraic groups.
\newblock Academic Press, Inc. 1985.

\bibitem{Landesman}
P. Landesman.
\newblock Generalized differential Galois theory.
\newblock Trans. Amer. Math.Soc. 360, pp.4441-4495, 2008.

\bibitem{LS1}
O. Le\'on S\'anchez.
\newblock Relative $D$-groups and differential Galois theory in several derivations.
\newblock Trans. Amer. Math, Soc., 367:7613--7638, 2015.

\bibitem{LSMP}
O. Le\'on S\'anchez, D. Meretzky and A. Pillay.
\newblock More on Galois cohomology, definability and differential algebraic groups.
\newblock The Journal of Symbolic Logic, 89(2):496--515, 2024.

\bibitem{LSPillay}
O. Le\'on S\'anchez and A. Pillay.
\newblock Some definable Galois theory and examples. 
\newblock The Bulletin of Symbolic Logic, 23(2):145--159, 2017.

%\bibitem{Ma}
%D. Marker.
%\newblock Embedding differential algebraic groups in algebraic groups.
%\newblock www.sci.ccny.cuny.edu/~ksda/PostedPapers/Marker050109.pdf

%\bibitem{Ma2}
%D. Marker.
%\newblock Model Theory: An Introduction.
%\newblock Springer-Verlag. New York, Inc, 2002. 

%\bibitem{Mc}
%T. McGrail.
%\newblock The model theory of differential fields with finitely many commuting derivations.
%\newblock The Journal of Symbolic Logic. Vol. 65, No. 2, pp.885-913, 2000.

\bibitem{Me}
D. Meretzky.
\newblock The short exact sequence in definable Galois cohomology. 
\newblock The Journal of Symbolic Logic. Published online 2025:1-15. doi:10.1017/jsl.2025.8


\bibitem{Me2}
D. Meretzky.
\newblock Differential Field Arithmetic.
\newblock University of Notre Dame. Thesis. https://doi.org/10.7274/28786154.v1

\bibitem{Pillay1}
A. Pillay.
\newblock Differential Galois theory I.
\newblock Illinois J. Math., 42(4):678--699, 1998.

\bibitem{Pillay2}
A. Pillay.
\newblock Algebraic D-groups and differential Galois theory. 
\newblock Pacific Journal of Mathematics. Vol. 216, No. 2, 2004.

\bibitem{Pillay3}
A. Pillay.
\newblock Remarks on Galois cohomology and definability.
\newblock The Journal of Symbolic Logic. Vol. 62, No. 2, pp. 487-492, 1997.

\bibitem{Pillay4}
A. Pillay.
\newblock Some foundational questions concerning differential algebraic groups.
\newblock Pacific Journal of Mathematics. Vol. 179, No.1, 1997.

\bibitem{Pillay5}
A. Pillay.
\newblock The Picard-Vessiot theory, constrained cohomology, and linear differential algebraic groups.
\newblock Journal de Math\'ematiques Pures et Appliqu\'ees. Vol. 108, pp.809-817, 2017.

\bibitem{PSokolovic}
A. Pillay and Z. Sokolovic. 
\newblock Superstable differential fields.
\newblock The Journal of Symbolic Logic. 57(1):97--108, 1992. 

%\bibitem{Po}
%B. Poizat.
%\newblock Stable Groups.
%\newblock Mathematical Surveys and Monographs Vol. 87. American Mathematical Society 2001.


%\bibitem{Seidenberg} 
%A. Seidenberg, 
%\newblock Contribution to the Picard-Vessiot theory of homogeneous linear differential equations.
%\newblock Amer. J. Math. 78 (1956), 808-817.

%\bibitem{So}
%S. Suer.
%\newblock Model theory of differentially closed fields with several commuting derivations.
%\newblock PhD Thesis. University of Illinois at Urbana-Champaign, 2007.

\bibitem{PuSi}
M. van der Put and M. Singer.
\newblock Galois theory of linear differential equations.
\newblock Springer-Verlag Berlin Heidelberg 2003






\end{thebibliography}
\end{document}